\documentclass[preprint,12pt]{amsart}

\usepackage{amssymb}
\usepackage{amsmath}
\usepackage{amsthm}
\usepackage{comment}
\usepackage{float}
\usepackage{hyperref}
\usepackage{graphicx}
\usepackage[margin=1in]{geometry}

\numberwithin{equation}{section}
\newcommand{\ddn}{{\mathrm{d}}}
\newcommand{\dd}{\,{\mathrm{d}}}
\newtheorem{theorem}{Theorem}[section]
\newtheorem{lemma}[theorem]{Lemma}

\newtheorem{remark}[theorem]{Remark}
\newtheorem{definition}[theorem]{Definition}

\begin{document}




\title[]{A Finite Volume Scheme for the Solution of a Mixed Discrete--Continuous Fragmentation Model}


\author{Graham Baird}

\author{Endre~S{\"u}li}

\address{Mathematical Institute, University of Oxford, Woodstock Road, Oxford OX2 6GG, UK
email: \texttt{graham.baird@maths.ox.ac.uk}, \texttt{endre.suli@maths.ox.ac.uk}}

\begin{abstract}
\noindent This paper concerns the construction and analysis of a numerical scheme for a mixed discrete-continuous fragmentation equation. A finite volume scheme is developed, based on a conservative formulation of a truncated version of the equations. The approximate solutions provided by this scheme are first shown to display conservation of mass and preservation of nonnegativity. Then, by utilising a Dunford--Pettis style argument, the sequence of approximate solutions generated is shown, under given restrictions on the model and the mesh, to converge (weakly) in an appropriate $L_1$ space to a weak solution to the problem. Additionally, by applying the methods and theory of operator semigroups, we are further able to show that weak solutions to the problem are unique and necessarily classical (differentiable) solutions. Finally, numerical simulations are performed to investigate the performance of the scheme and assess its rate of convergence.
\end{abstract}




\maketitle


\section{Introduction}

Fragmentation and coagulation processes occur in many physical systems, with the associated mathematical models receiving much attention in the literature. Example application areas include colloid science \cite{ziff80,Costas95}, population dynamics \cite{Degond17a,Degond17b} and astrophysics \cite{johansen08,Dullemond14}. However, analytical solutions to these models are only available for a limited number of specific cases, and we often have to resort to approximate solutions generated by an appropriate numerical scheme. A range of numerical techniques have been applied to these problems, and these broadly fall into two categories: those involving a stochastic (Monte Carlo) element, for example \cite{guias1997,babovsky99,eibeck2000} and those based around various deterministic approximation schemes \cite{Kumar96, Nicmanis96, Barrett96, Smirnov16}. The introductory chapter of \cite{kumarphd10} and the references therein provide a detailed overview of a number of these approximation methods.

In the earlier work \cite{baird18}, we presented a mixed discrete-continuous model of fragmentation in an attempt to resolve the issue of `shattering' mass-loss observed in some purely continuous models \cite{mcgrady87}. By modelling the mass distribution amongst the smallest particles using a discrete model, whilst modelling the distribution of larger particle masses with a continuous model, the aim was to introduce a higher degree of physical fidelity thus resolving the shattering mass-loss problem, whilst also retaining the mathematical efficiency of the continuous model.

Given the similarities between this model and those existing in the literature, in addition to the added complexity of the mixed framework, we would expect in most cases to have to rely on numerical methods to obtain a solution. In this paper we present a numerical scheme for the solution of the mixed fragmentation model proposed in \cite{baird18}. The basis of the scheme is a finite volume discretisation of the continuous regime equation. The use of such a method would appear a reasonable choice in this case, given its conservative nature and the motivation behind the model development. Indeed, finite volume schemes have been commonly applied to the solution of coagulation and fragmentation equations, with the first such use being \cite{filbet04}, where the case of pure coagulation was considered. For problems involving fragmentation, the article \cite{bourgade08} sees such a scheme employed in approximating the binary coagulation and fragmentation equation, whilst \cite{kumarphd10} and \cite{Kumar15} examine their use for the multiple fragmentation equation, with \cite{kumar14} extending this to include coagulation. Further works have seen these methods applied to a number of  coagulation--fragmentation model variants, for example with the inclusion of spacial diffusion \cite{filbet08diffusion} and additional nucleation and growth processes \cite{QAMAR2009, kumar2013moment}. Whilst a number of articles  \cite{QAMAR2007,forestier2012,Saha19} cover the approximation of multi-dimensional coagulation or fragmentation, whereby particles may be classified by additional variables beyond their mass or volume.

\subsection{Mixed Discrete-Continuous Model}
In the mixed model of \cite{baird18}, a cut-off value $N\in \mathbb{N}$ is introduced; above this cut-off, particle mass is considered as a continuous variable, whilst below it, the particles are forced to take discrete integer masses. If we denote by $u_C(x,t)$ the particle mass density within the continuous mass regime ($x>N$), then the evolution of $u_C(x,t)$ is governed by the continuous multiple fragmentation equation:
\begin{align}\label{equation301}
\hspace{-8mm}\frac{\partial u_C(x,t)}{\partial t}&=-a(x)u_C(x,t)+\int_{x}^{\infty}a(y)b(x|y)u_C(y,t)\dd y, \hspace{1.92mm}x>N,\hspace{1.2mm}t>0,\\
                            u_C(x,0)&=c_{0}(x).\nonumber
\end{align}
This equation is similar in form to the multiple fragmentation equation introduced in \cite{mcgrady87}. The function $a(x)$ provides the fragmentation rate for a particle of mass $x$, whilst $b(x|y)$ represents the distribution of particles of mass $x>N$ resulting from the break-up of a particle of mass $y>x$. The functions $a$ and $b$ are assumed to be nonnegative measurable functions, defined on $\left(N,\infty\right)$ and $\left(N,\infty\right)\times\left(N,\infty\right)$, respectively. We also require $b(x|y)=0$ for $x>y$, since no particle resulting from a fragmentation event can have a mass exceeding the original particle. The initial mass distribution within the continuous regime is given by the nonnegative function $c_{0}(x)$.

Letting $u_{Di}(t)$ denote the concentration of discrete mass $i$-mer particles  ($i\leq N$) and $u_{D}(t)$ the $N$-component vector taking these values as entries, the change in the values $u_{Di}(t)$, $i=1,\dots,N$, is governed by the equation:

\begin{align}\label{equation302}
\hspace{-80mm}\frac{\ddn u_{Di}(t)}{\ddn t} &\hspace{-.7mm}=\hspace{-.7mm}-a_iu_{Di}(t)+\hspace{-1.9mm}\sum_{j=i+1}^{N}\hspace{-1.2mm}a_jb_{i,j}u_{Dj}(t)
+\hspace{-1mm}\int_{N}^{\infty}a(y)b_i(y)u_{\small{C}}(y,t)\dd y,\hspace{.7mm}\hspace{1.1mm}t>0,\\
 u_{D}(0)&=d_{0}.\nonumber
\end{align}
In the case of $i=N$, the second term becomes an empty sum and is taken to be $0$. The values $a_i$ give the rates at which $i$-mer particles fragment, with $a_1=0$. The quantities $b_{i,j}$ give the expected number of $i$-mers produced from the fragmentation of a $j$-mer and the functions $b_i(y)$ give the expected number of $i$-mers produced from the fragmentation of a particle of mass $y>N$. The underlying physics demands that each $a_i$, $b_{i,j}$ and $b_i(y)$ be nonnegative. Finally, $d_{0}$ is the $N$-component vector of nonnegative values, specifying the initial concentrations within the discrete regime.

During each fragmentation event, mass is simply redistributed from the larger particle to the smaller resulting particles, but the total mass involved should be conserved. This gives rise to the following two conditions to supplement equations~\eqref{equation301} and~\eqref{equation302}:
\begin{align}
\label{equation303}&\int_{N}^{y}xb(x|y)\dd x+\sum_{j=1}^{N}jb_j(y)=y\hspace{2mm} \text{for} \hspace{2mm}y>N,\\
\label{equation304}&\sum_{j=1}^{i-1}jb_{j,i}=i\hspace{2mm} \text{for} \hspace{2mm}i=2,\ldots,N.
\end{align}
The condition~\eqref{equation303} is an expression of mass conservation upon the fragmentation of a particle from the continuous mass regime. The equation~\eqref{equation304} comes from the conservation of mass when a particle from the discrete mass regime breaks up.

For further details on the mixed discrete-continuous model and its properties, the reader is directed to consult \cite{baird18} or \cite{bairdphd}.

\subsection{Truncation and Reformulation}
When considering the numerical solution of equations~\eqref{equation301} and~\eqref{equation302}, we encounter an issue in that the range of the continuous mass variable $x$ is an unbounded interval, which presents a computational problem. We therefore introduce a truncation parameter $R>N$, and restrict the continuous mass variable to the range $N<x<R$. Therefore, in place of equations~\eqref{equation301}, we consider the truncated version
\begin{align}\label{equation40001}
\hspace{-8mm}\frac{\partial u_{C}^R(x,t)}{\partial t}&=-a(x)u_{C}^R(x,t)+\int_{x}^{R}a(y)b(x|y)u_{C}^R(y,t)\dd y, \hspace{1.92mm}N<x<R,\hspace{1.2mm}t>0,\\
                            u_{C}^R(x,0)&=\chi_{(N,R)}(x)c_{0}(x), \hspace{2.5mm}N<x<R,\nonumber
\end{align}\\
where $\chi_{(N,R)}$ denotes the characteristic function of the interval $(N,R)$.
\noindent Taking our lead from the aforementioned articles, we now rewrite equation \eqref{equation40001} in a conservative form, although in our case we must include an additional sink term to account for the mass leaked down to the discrete regime. Therefore we end up with the following equation for the mass quantity $xu_{C}^R(x,t)$:
\begin{equation}\label{equation402}
\frac{\partial \left(xu_{C}^R\right)}{\partial t}=\frac{\partial \mathcal{F}^R\left(xu_{C}^R\right)}{\partial x}-S(xu_{C}^{R}),\hspace{1.5mm}
u_{C}^R(x,0)=c_{0}(x),\hspace{1.5mm}\text{for}\hspace{1.5mm}N<\hspace{-.5mm}x\hspace{-.5mm}
< R,\hspace{.5mm}t>0,
\end{equation}
where $\mathcal{F}^R$ and $S$ are a truncated flux term and sink term, respectively, given by
\begin{equation*}
\hspace{-4mm}\mathcal{F}^R(f)=\int_{x}^{R} \int_{N}^{x} \frac{y}{z} a(z) b(y|z)f(z)\dd y\dd z,\hspace{4mm}S(f)=\frac{a(x)}{x}\sum_{i=1}^{N}ib_i(x)f(x),\hspace{2mm}\text{for}\hspace{2mm}N<x<R.
\end{equation*}
The equation \eqref{equation40001} may be recovered from \eqref{equation402} by a formal application of Leibniz's rule for differentiating under the integral. However, the equivalence of the two forms can be seen to be justified rigorously in \cite[Appendix C]{bairdphd}. Before continuing, we establish a result concerning the behaviour of the flux term $\mathcal{F}^R$ at the limits of our domain.
\begin{lemma}\label{lastthing}If the kernels $a$ and $b$ are assumed to belong to $L_{\infty,loc}$ on the domains $[N,\infty)$ and $[N,\infty)\times[N,\infty)$ respectively, which will be the case in the upcoming analysis, then for
$f\in L_1(N,R)$ the flux term $\mathcal{F}^R(f)$ satisfies
\[\lim_{x\rightarrow N,R}\left|\mathcal{F}^R(f)(x)\right|=0.\]
\end{lemma}
\begin{proof}It is a straightforward matter to bound $\mathcal{F}^R(f)$ as follows:
\begin{equation}\label{final}
\left|\mathcal{F}^R(f)(x)\right|\leq \int_N^R \chi_{(x,R)}(z)\frac{a(z)\left|f(z)\right|}{z}\left(\int_N^x yb(y|z)\dd y \right)\dd z,
\end{equation}
which holds for $x \in(N,R)$. Recalling the mass conservation condition~\eqref{equation303}, we deduce that
\begin{equation*}
\chi_{(x,R)}(z)\left( \int_N^x yb(y|z)\dd y \right)\leq \int_N^z yb(y|z)\dd y\leq z,
\end{equation*}
for all $z\in(N,R)$. Hence the integrand appearing in~\eqref{final} is bounded above by $a(z)\left|f(z)\right|$, which, thanks to
$a\in L_{\infty,loc}[N,\infty)$ and $f\in L_1(N,R)$, is integrable.

Considering the limit as $x\rightarrow N$ first, if we denote by $\beta(R)$ the essential supremum of $b$ over $[N,R]\times[N,R]$, then we have
\begin{equation*}
\chi_{(x,R)}(z)\left( \int_N^x yb(y|z)\dd y \right)\leq x \beta(R)(x-N).
\end{equation*}
As such, the integrand in~\eqref{final} converges pointwise to 0 over $z\in(N,R)$ as we let $x\searrow N$. An application of the Lebesgue
dominated convergence theorem then gives the required convergence of $\left|\mathcal{F}^R(f)(x)\right|$ as $x\searrow N$. Turning now to the limit as $x\nearrow R$, another application of condition~\eqref{equation303} provides us with
\begin{equation*}
\chi_{(x,R)}(z)\left( \int_N^x yb(y|z)\dd y \right)\leq\chi_{(x,R)}(z) \int_N^z yb(y|z)\dd y\leq \chi_{(x,R)}(z)z,
\end{equation*}
for $z\in(N,R)$. Therefore, the integrand from~\eqref{final} must again converge pointwise to 0 over $(N,R)$, this time as we let $x\nearrow R$. Another application of the Lebesgue dominated convergence theorem gives the convergence of $\left|\mathcal{F}^R(f)(x)\right|$ to 0, as $x\nearrow R$.
\end{proof}
This result will be utilised later in a number of arguments, most significantly in approximating $\mathcal{F}^R$ within our numerical scheme and in establishing a weak formulation of equation~\eqref{equation402}.

The truncation of the continuous mass interval also has an impact on our discrete regime equation; therefore, instead of equation \eqref{equation302}, we consider
\begin{align}\label{equation403}
\hspace{-80mm}\frac{\ddn u_{Di}^R(t)}{\ddn t}
&\hspace{-.7mm}=\hspace{-.7mm}-a_iu_{Di}^R(t)+\hspace{-1.9mm}\sum_{j=i+1}^{N}\hspace{-1.2mm}a_jb_{i,j}
u_{Dj}^R(t)+\hspace{-1mm}\int_{N}^{R}\hspace{-.5mm}a(y)b_i(y)u_{C}^R(y,t)\dd y,\\
u_{Di}^R(0)&={d_{0}}_i,\hspace{2mm}\text{for}\hspace{2mm}i=1,2,\ldots,N,\hspace{2mm}t>0.\nonumber
\end{align}
In the case of $i=N$, the empty sum above is taken to be zero; this convention will be adopted in all similar cases which follow.

This truncation procedure is a standard approach when dealing with fragmentation and coagulation problems, having been applied for example in \cite{mcl1}, where the theory and methods of operator semigroups were employed, and \cite{stewart89} where an alternative weak compactness style argument was adopted. The common approach of these works involves establishing the existence of solutions to a sequence of such truncated problems. A limit is then obtained as the truncation point is increased without bound, with this limit then being shown to satisfy the untruncated problem in some sense. Although in this article we restrict our attention to the numerical approximation of the truncated discrete--continuous problem, as given by equations \eqref{equation40001} and \eqref{equation403}, it can be shown that the solutions to the truncated problems converge, in an appropriate space, to give the solutions to the untruncated \eqref{equation301} and \eqref{equation302}. The proof of this convergence argument follows similar lines to that set out in \cite[Section 8.3.2]{banasiak06}, with the reader being directed to \cite[Chapter 6]{bairdphd} for the specific details. Additionally, the reader may find an empirical examination of this convergence in \cite[Section 7.5]{bairdphd}, where the key factors influencing the convergence, and thus the selection of an appropriate $R$ are investigated.

\section{Preliminaries}

Having set out our problem in the previous section, we now present a brief outline of the key results which appear in the upcoming material and which may be considered nonstandard or which are particular to our case.

\begin{theorem}\label{theorem201} In the analysis pursued in subsequent results, we shall be working extensively in spaces of the type $L_1$. In particular we shall be working in the spaces $L_1=L_1((N,R)\times [0,T),\dd x \dd t)$ and $L_1^1=L_1((N,R)\times[0,T),x \dd x \dd t)$, where
$N$ is a positive integer and $R>N$ is a finite real value. With the associated norms, these form equivalent spaces.
\end{theorem}
\begin{proof}
First let us suppose that $f\in L_1^1$; then we have
\[\|f\|_{L_1}=\int_0^T\int_N^R\left|f(x,t)\right|\dd x \dd t\leq\frac{1}{N}\int_0^T\int_N^R\left|f(x,t)\right|\,x\dd x \dd t=\frac{1}{N}\|f\|_{L_1^1}.\]
Therefore $f \in L_1$ also, with $\|f\|_{L_1}\leq \frac{1}{N}\|f\|_{L_1^1}$. Now let us assume that $f\in L_1$; then we have
\[\|f\|_{L_1^1}=\int_0^T\int_N^R\left|f(x,t)\right|\,x \dd x \dd t\leq R\int_0^T\int_N^R\left|f(x,t)\right|\dd x \dd t=R\|f\|_{L_1}.\]
Hence $f \in L_1^1$ with $\|f\|_{L_1^1}\leq R\|f\|_{L_1}$. Taken together, the above results show us that the spaces $L_1((N,R)\times[0,T),\dd x \dd t)$
and $L_1((N,R)\times[0,t),x\dd x \dd t)$ contain the same elements and have equivalent norms.
\end{proof}
This result shall prove useful in the forthcoming analysis, allowing us to switch spaces when mathematically convenient whilst retaining convergence.

Given a sequence $\left\lbrace f_n\right\rbrace_{n=1}^\infty$ in a normed vector space $\left(X,\|\cdot\|\right)$, we assume the reader is familiar with the concept of \emph{weak} convergence and in particular its definition in spaces of the type $L_1(\Omega,\mu)$. In our analysis when handling weakly convergent sequences we will usually find them appearing alongside other factors and we would like the product to converge weakly also. The following theorem gives us sufficient conditions for the product of two sequences to converge weakly and will be used extensively in the convergence proofs for our numerical schemes.

\begin{theorem}\label{theorem203}Let $\left(\Omega,\mathcal{A},\mu\right)$ be a measure space with $\mu$ finite. Suppose $f_h\rightharpoonup f$ in $L_1\left(\Omega,\mu\right)$, $g_h\rightarrow g$ point-wise $\mu$ a.e. in $\Omega$, and $\sup_{h}\|g_h\|_{L_\infty}<\infty$, then $f_hg_h\rightharpoonup fg$ in $L_1\left(\Omega,\mu\right)$.
\end{theorem}
\begin{proof}The reader is referred to \cite[Proposition 2.61]{fonseca07}.
\end{proof}

The main part of our convergence argument utilises the Dunford--Pettis theorem, which provides us with sufficient conditions to establish the weak convergence of our sequence of approximations. One such condition is that of \emph{equiintegrability}. There are a number of equivalent characterisations of equiintegrability, which the reader may find in \cite[Theorem 2.29]{fonseca07}. For our purposes the most important characterisation of equiintegrability is given by de la Vall\'ee Poussin's theorem, a refined version of which is given below.

\begin{theorem}\label{theorem204}\textnormal{(de la Vall\'ee Poussin's Theorem)} Let $\mathcal{F}$ be a bounded subset of $L_1\left(\Omega,\mu\right)$, then $\mathcal{F}$ is equiintegrable if and only if there exists a nonnegative, convex function $\Phi\in C^\infty \left([0,\infty)\right)$, with $\Phi(0)=0$ and $\Phi'(0)=1$, such that $\Phi'$ is concave and
\begin{equation*}
\frac{\Phi(x)}{x}\rightarrow\infty\hspace{2mm}\text{as}\hspace{2mm}x\rightarrow\infty \hspace{3mm}\text{and}\hspace{3mm}
\sup_{f\in \mathcal{F}}\int_{\Omega}\Phi\left(\left|f\right|\right)\dd \mu<\infty.
\end{equation*}
\end{theorem}
\begin{proof} The necessity of this condition can be derived easily from \cite[Theorem 8]{laurencot15}, which under the assumption that $\mathcal{F}$ is equiintegrable provides us with a $\Psi$ satisfying all the stated conditions with the exception that the function $\Psi$ has derivative $0$ at $0$ and is not stated to be nonnegative. Given such a $\Psi$, we set $\Phi(x)=\Psi(x)+x$. Then $\Phi$ retains the required properties of $\Psi$ but additionally $\Phi'(0)=1$. Also, by utilising the following standard inequality for $C^1$ convex functions
\begin{equation}\label{equation203}
\Phi(x)\geq\Phi(y)+\Phi'(y)(x-y),
\end{equation}
with $x\geq 0$ and $y=0$  we can see that $\Phi(x)$ must be nonnegative on $[0,\infty)$. The sufficiency of our conditions comes straight from the standard version of the de la Vall\'ee Poussin theorem \cite[Theorem 2.29 (iii)]{fonseca07}.
\end{proof}

In our analysis we shall require some properties of such a function, which we set out in the following lemma.
\begin{lemma}\label{lemma203} Let $\Phi$ be as in Theorem~\ref{theorem204}; then for nonnegative $x$ and $y$ we have the following:
\begin{enumerate}
  \item $x\Phi'(y)\leq \Phi(x)+\Phi(y)$,
  \item $\Phi'(y)\geq 0$.
\end{enumerate}
\end{lemma}
\begin{proof}The first of these inequalities is nonstandard and the proof can be found in \cite[Proposition 13 (30)]{laurencot15}. For the second property we return to inequality~\eqref{equation203}, with $x=0$ and $y\geq 0$, which gives us
\[\underbrace{\Phi(0)}_{=0}\geq\underbrace{\Phi(y)}_{\geq 0}+\Phi'(y)(0-y).\]
An obvious rearrangement yields
\[y\Phi'(y)\geq\Phi(y)\geq0.\]
Now in the case that $y=0$ property (ii) is given by the definition of $\Phi$. Hence we may assume that $y>0$ and divide through by it to obtain the desired result that $\Phi'(y)\geq 0$.
\end{proof}
We now come to the Dunford--Pettis theorem, one of the most significant technical tools applied in this work. The theorem provides necessary and sufficient conditions for a subset of an $L_1$ space to be \emph{weakly sequentially compact}. That is, any sequence in the subset must have a subsequence which is weakly convergent.

\begin{theorem}\label{theorem205}\textnormal{(Dunford--Pettis Theorem)} Let $\left(\Omega,\mathcal{A},\mu\right)$ be a measure space and let $\mathcal{F}\subset L_1\left(\Omega,\mu\right)$. Then $\mathcal{F}$ is weakly sequentially compact if and only if the following conditions are satisfied:
\begin{enumerate}
  \item $\mathcal{F}$ is bounded in $L_1\left(\Omega,\mu\right)$;
  \item $\mathcal{F}$ is equiintegrable;
  \item For every $\varepsilon >0$ there exists $A_\varepsilon\subset \Omega$ with $A_\varepsilon\in \mathcal{A}$ such that $\mu\left(A_\varepsilon\right)<\infty$ and
      \[ \sup_{f\in \mathcal{F}} \int_{\Omega\setminus A_\varepsilon} \left|f\right|\dd \mu\leq \varepsilon.\]
\end{enumerate}
We note that in the case that $\mu(\Omega)<\infty$ condition \textup{(}iii\textup{)} is automatically satisfied by taking $A_\varepsilon= \Omega$ for all values of $\varepsilon$.
\end{theorem}
\begin{proof}See \cite[Theorem 2.54]{fonseca07}.
\end{proof}

In the later analysis of this paper we shall be relying heavily on the methods and theory of operator semigroups. In particular the concept of substochastic semigroups, the Kato--Voigt perturbation theorem and the notion of semigroup honesty. For the sake of brevity we refrain from outlining such material here, however the reader may find details of the requisite results in the preliminary sections of \cite{baird18} or \cite{bairdphd} or the text \cite{banasiak06}.

\section{Development of the Numerical Scheme}\label{section1}

We now introduce our numerical approximation scheme for the truncated system, \eqref{equation402} and \eqref{equation403}. First we must
discretise the continuous mass variable $x$, and so we introduce the mesh $\left\lbrace x_{i-1/2}\right\rbrace_{i=0}^{I_h}$ on the interval $(N,R)$, with
\[x_{-1/2}\hspace{-.2mm}=\hspace{-.2mm}N,\hspace{1.3mm}x_{I_h-1/2}\hspace{-.2mm}=\hspace{-.2mm}R,\hspace{1.3mm}x_{i}\hspace{-.2mm}=\hspace{-.2mm}
(x_{i-1/2}+x_{i+1/2})/2,
\hspace{1.3mm}h/k<\Delta x_{i}\hspace{-.2mm}=\hspace{-.2mm}x_{i+1/2}-x_{i-1/2}\hspace{-.5mm}<\hspace{-.5mm}h,\]
where $h\in(0,1)$ and $k>1$ is some constant. Additionally we denote the interval $[x_{i-1/2},x_{i+1/2})$ by $\Lambda_i$, however the (left-hand-most)
interval $\Lambda_0$ is taken to be $(x_{-1/2},x_{1/2})$.

For the time variable $t$, if $T$ is the final time up to which we wish to compute an approximate solution, then we define the time step
$\Delta t=T/M$ where $M$ is some large integer. The time points are then given by $t_n=n\Delta t$ for $n=0,1,\dots,M$ with corresponding time intervals
$\tau_n=[t_n,t_{n+1})$ for $n=0,1,\dots,M-1$.

We restrict the choice of the mesh by assuming the existence of positive constants $k_1$ and $k_2$ so that the mesh sizes $h$ and $\Delta t$ satisfy
\begin{equation}\label{equation404}
k_1h\leq \Delta t \leq k_2 h.
\end{equation}

The numerical scheme requires representative values for the functions $a(x)$, $b(x|y)$ and $b_i(y)$ over the appropriate intervals. This is
done by taking their average value over each interval. Therefore we define
\begin{equation}\label{equation405}
A_i=\frac{1}{\Delta x_{i}}\int_ {\Lambda_{i}} a(x)\dd x\hspace{3mm}\text{for}\hspace{2mm}i=0,1,\dots,I_h-1,\nonumber
\end{equation}
as our approximation of $a(x)$ over the interval $\Lambda_{i}$. We approximate $b(x|y)$ over $\Lambda_{i}\times \Lambda_{j}$ by
\begin{equation}\label{equation406}
B_{i,j}=\frac{1}{\Delta x_{i}\Delta x_{j}}\int_ {\Lambda_{j}} \int_ {\Lambda_{i}}\hspace{-1mm} b(x|y)\dd x\dd y\hspace{2mm}\text{for}\hspace{2mm}
i=0,1,...,I_h-1\hspace{2mm}\text{and}\hspace{2mm}j=0,1,\ldots,I_h-1,\nonumber
\end{equation}
and the functions $b_i(y)$ are approximated over $\Lambda_{j}$ by the values
\begin{equation}\label{equation407}
\tilde{B}_{i,j}=\frac{1}{\Delta x_{j}}\int_ {\Lambda_{j}} b_i(y)\dd y\hspace{2mm}\text{for}\hspace{2mm}i=1,2,\ldots,N\hspace{2mm}\text{and}\hspace{2mm}
j=0,1,\dots,I_h-1.\nonumber
\end{equation}

We note by our initial assumption regarding the nonnegativity of $a$, $b$ and $b_i$, that each of the values introduced above must be nonnegative. If $\chi_I$ denotes the characteristic function of a set $I$, then we can construct piecewise constant approximations to the functions $a$, $b$ and $b_i$ as follows:
\begin{equation*}
a^h(x) \hspace{-.5mm}=\hspace{-.5mm}\sum_{i=0}^{I_h-1} \chi_{\Lambda_i}(x)A_i, \hspace{1.5mm}
b^h(x|y)\hspace{-.5mm}=\hspace{-.5mm}\sum_{j=0}^{I_h-1}\sum_{i=0}^{I_h-1}\chi_{\Lambda_i}(x)\chi_{\Lambda_j}(y)B_{i,j}, \hspace{1.5mm}
b_i^h(y)\hspace{-.5mm}=\hspace{-.5mm}\sum_{j=0}^{I_h-1} \chi_{\Lambda_j}(y)\tilde{B}_{i,j}.
\end{equation*}
\begin{remark} \label{convremark1}
This is a standard means of approximation and assuming the choice of kernels is suitably restricted, the approximations will converge pointwise to the desired functions almost everywhere on the appropriate domains. In our case, the kernels $a$ and $b$ will be assumed to be $L_{\infty,loc}$ on $[N,\infty)$ and $[N,\infty)\times[N,\infty)$, respectively. In addition, the restriction \eqref{equation303} determines each $b_i$ as an element of $L_{\infty,loc}[N,\infty)$. Having $a$, $b$ and $b_i$ as $L_{\infty,loc}$ functions is sufficient to ensure that the approximations $a^h$, $b^h$ and $b_i^h$ converge pointwise to  $a$, $b$ and $b_i$ almost everywhere in their respective domains. This is a standard result, however full details can be found in \cite[Lemma 4.2.1]{bairdphd}.
\end{remark}

We are now ready to construct the approximation scheme. Let $u_{C}^{n,i}$ denote our approximation to $u_C^R(x,t)$ over the mass interval $\Lambda_i$ for the time interval $\tau_n$. The equation \eqref{equation402} is then approximated by

\begin{equation}\label{equation4075}
x_i\frac{u_{C}^{n+1,i}-u_{C}^{n,i}}{\Delta t}=\frac{F_{i+1/2}^n-F_{i-1/2}^n}{\Delta x_{i}}-S_i^n,\nonumber
\end{equation}

where $F_{i-1/2}^n$ is an approximation of the flux $\mathcal{F}^R(xu_C^R)$ at the point $x=x_{i-1/2}$ over the time interval $\tau_n$, and is given by
\begin{align}\label{equation408}
\left(\mathcal{F}^R(xu_C^R)\right)(x_{i-1/2})&=\int_{x_{i-1/2}}^{R} \int_{N}^{x_{i-1/2}}y a(z) b(y|z)u_C^R(z,t)\dd y\dd z\nonumber\\
&=\sum_{j=i}^{I_h-1}\int_{\Lambda_j}\left(\sum_{k=0}^{i-1}\int_{\Lambda_k} y a(z) b(y|z)u_C^R(z,t)\dd y\right)\dd z\nonumber\\
&\approx\sum_{j=i}^{I_h-1}\sum_{k=0}^{i-1}x_{k}A_jB_{k,j}u_{C}^{n,j}\Delta x_{k}\Delta x_{j}=:F_{i-1/2}^n\nonumber,
\end{align}
for $i=1,\ldots,I_h-1$, with $F_{-1/2}^n=F_{I_h-1/2}^n=0$, which can be justified by Lemma~\ref{lastthing}. The values $S_i^n$ approximate the sink term $S(xu_C^R)$ over $\Lambda_i$ for the
time interval $\tau_n$, and are computed by
\begin{equation}\label{sink}
S_i^n=A_i\sum_{j=1}^{N}j\tilde{B}_{j,i}u_{C}^{n,i}\hspace{1.5mm}\text{for}\hspace{1.5mm}i=0,1,\dots,I_h-1.
\end{equation}

This gives rise to the following numerical method for the computation of the approximations $u_{C}^{n,i}$:
\begin{equation}\label{equation409}
u_{C}^{n+1,i}=u_{C}^{n,i}+\frac{\Delta t}{x_i\Delta x_{i}}(F_{i+1/2}^n-F_{i-1/2}^n)-\frac{\Delta t}{x_i}S_i^n\hspace{2mm}\text{for}
\hspace{1.5mm}\left\lbrace \begin{array}{l} \hspace{.5mm}i=0,1,\dots,I_h-1, \\ n=0,1,\dots,M-1. \end{array}\right.
\end{equation}
The sequence of approximations generated by \eqref{equation409} requires us to provide an initial set of values to get started. For our starting values we simply average the initial datum over each of the mass intervals; hence
\begin{equation}\label{equation410}
u_{C}^{0,i}=\frac{1}{\Delta x_{i}}\int_ {\Lambda_{i}} c_0(x)\dd x\hspace{3mm}\text{for}\hspace{2mm}i=0,1,\dots,I_h-1.
\end{equation}
Then our approximation to $u_{C}^R(x,t)$ over $(N,R)\times[0,T)$ is constructed as follows:
\begin{equation}\label{equation4105}
u^h_C(x,t)=\sum_{n=0}^{M-1} \sum_{i=0}^{I_h-1} \chi_{\Lambda_i}(x)\chi_{\tau_n}(t)u_{C}^{n,i}.
\end{equation}
\begin{remark} \label{convremark2}
The convergence proof for our numerical scheme requires the initial approximation given by \eqref{equation410} and \eqref{equation4105} to converge strongly in $L_1(N,R)$ to the restriction of $c_0$ to $(N,R)$. Again, that this is the case with our definition of the discretised initial datum is a standard result and details can be found in \cite[Lemma 4.2.2]{bairdphd}.\\
\end{remark}

Now considering the discrete regime, let $u_{D}^{n,i}$ denote our approximation of $u_{Di}^R(t)$ over the time interval $\tau_n$. Equation \eqref{equation403} is then
approximated as\\
\begin{equation}\label{equation411}
\frac{u_{D}^{n+1,i}-u_{D}^{n,i}}{\Delta t}=-a_iu_{D}^{n,i}+\hspace{-1.9mm}\sum_{j=i+1}^{N}a_jb_{i,j}u_{D}^{n,j}+\sum_{j=0}^{I_h-1}A_j
\tilde{B}_{i,j}u_{C}^{n,j}
\Delta x_{j},\nonumber
\end{equation}
giving rise to the relation
\begin{equation}\label{equation412}
u_{D}^{n+1,i}=(1-\Delta ta_i)u_{D}^{n,i}+\Delta t\sum_{j=i+1}^{N}\hspace{-1mm}a_jb_{i,j}u_{D}^{n,j}+\Delta t\sum_{j=0}^{I_h-1}\hspace{-1mm}A_j\tilde{B}_{i,j}u_{C}^{n,j}
\Delta x_{j}\hspace{2mm}\text{for}\hspace{1.5mm}\left\lbrace \begin{array}{l} \hspace{.5mm}i=1,\dots,N, \\\hspace{-.5mm} n=0,1,\dots,M-1. \end{array}\right.
\end{equation}
The initial values for the discrete approximation are simply given by the initial condition vector $d_0$, so that $u_{D}^{0,i}={d_0}_i$
for $i=1,\ldots, N$. Then our approximations $u_{Di}^h(t)$ to $u_{Di}^R(t)$ for $t \in [0,T)$ are given by
\begin{equation}\label{equation4125}
u_{Di}^h(t)=\sum_{n=0}^{M-1} \chi_{\tau_n}(t)u_{D}^{n,i}\hspace{3mm}\text{for}\hspace{2mm}i=1,2,\ldots,N.
\end{equation}

\section{Properties of Numerical Solutions: Nonnegativity and Mass Conservation}

In the article \cite{baird18} we proved the existence and uniqueness of a solution to our system \eqref{equation301} and \eqref{equation302}. This solution was shown to possess a number of properties that we would expect given the physical nature of the model, namely the solution preserved nonnegativity and conserved total mass. In the following sections we examine whether the approximate solution provided by \eqref{equation4105} and \eqref{equation4125}, also displays these properties. These properties, apart from being physically relevant, will also be utilised in the forthcoming proofs of the convergence of the approximations \eqref{equation4105} and \eqref{equation4125} to a solution to the system \eqref{equation40001} and \eqref{equation403}, and subsequently the uniqueness and differentiability of that solution.

\subsection{Nonnegativity of the Numerical Solution}

\begin{lemma}\label{lemma401} For a fixed partition $(x_{i-1/2})_{i=0}^{I_h}$, suppose that $\Delta t$ is sufficiently small
that the following condition is satisfied:
\begin{equation*}
0<\Delta t \leq \frac{x_i}{A_i\left(\sum_{k=0}^{i-1}x_{k}B_{k,i}\Delta x_{k}+\sum_{j=1}^{N}j\tilde{B}_{j,i}\right)},
\end{equation*}
for all $i \in \left\lbrace0,1,\dots,I_h-1\right\rbrace$ such that the denominator is nonzero, and
\begin{equation*}
0<\Delta t \leq \frac{1}{a_i},
\end{equation*}
for all $i \in \left \lbrace 2,\dots,N \right\rbrace$ such that $a_i\neq0$. Then, the approximate solutions defined
by \eqref{equation4105} and \eqref{equation4125} preserve nonnegativity.
\end{lemma}

\begin{proof} Starting with the approximation for the continuous regime, let us consider equation \eqref{equation409}. By cancelling common terms we
get that
\begin{align}\label{equation413}
&F_{i+1/2}^n-F_{i-1/2}^n=x_{i}\Delta x_{i}\sum_{j=i+1}^{I_h-1}A_jB_{i,j}u_{C}^{n,j}\Delta x_{j}-A_iu_{C}^{n,i}\Delta x_{i}\sum_{k=0}^{i-1}x_{k}B_{k,i}
\Delta x_{k}
\end{align}
for $i=1,\dots,I_h-2$. Therefore we have

\begin{align*}
&\frac{\Delta t}{x_i\Delta x_{i}}(F_{i+1/2}^n-F_{i-1/2}^n)-\frac{\Delta t}{x_i}S_i\\
&=\Delta t\sum_{j=i+1}^{I_h-1}A_jB_{i,j}u_{C}^{n,j}\Delta x_{j}-\frac{\Delta t}{x_i}A_iu_{C}^{n,i}\sum_{k=0}^{i-1}x_{k}B_{k,i}\Delta x_{k}
-\frac{\Delta t}
{x_i}A_iu_{C}^{n,i}\sum_{j=1}^{N}j\tilde{B}_{j,i} \\
&=\Delta t\sum_{j=i+1}^{I_h-1}A_jB_{i,j}u_{C}^{n,j}\Delta x_{j}-\frac{\Delta t}{x_i}A_iu_{C}^{n,i}\left(\sum_{k=0}^{i-1}x_{k}B_{k,i}\Delta x_{k}+
\sum_{j=1}^{N}j\tilde{B}_{j,i}\right).
\end{align*}\\
Substituting this into \eqref{equation409} gives us
\begin{align}\label{equation901}
u_{C}^{n+1,i}&=\left(1-\frac{\Delta t}{x_i}A_i\left(\sum_{k=0}^{i-1}x_{k}B_{k,i}\Delta x_{k}+\sum_{j=1}^{N}j\tilde{B}_{j,i}\right)\right)u_{C}^{n,i}
+\Delta t\sum_{j=i+1}^{I_h-1}A_jB_{i,j}u_{C}^{n,j}\Delta x_{j},
\end{align}
for $i=1,\dots,I_h-2$. The cases $i=0$ and $i=I_h-1$ can be handled similarly to obtain the same result, where the empty sums are taken
as 0.

From this it is clear that if each of the approximations $u_{C}^{n,i}$ is nonnegative, and provided $\Delta t$ is sufficiently small
such that the term within the outer brackets is nonnegative, then each of the approximations $u_{C}^{n+1,i}$, for the subsequent time step, will also be
nonnegative. Hence to ensure the approximations $u_{C}^{n+1,i}$ are nonnegative we can take\\
\begin{equation}\label{equation414}
0<\Delta t \leq \frac{x_i}{A_i\left(\sum_{k=0}^{i-1}x_{k}B_{k,i}\Delta x_{k}+\sum_{j=1}^{N}j\tilde{B}_{j,i}\right)}\hspace{3mm}\text{for}
\hspace{2mm}i=0,1,\dots,I_h-1.
\end{equation}
In the case of the above denominator being zero for some $i$, such that the bound \eqref{equation414} is undefined, then $u_{C}^{n+1,i}$ can be seen from \eqref{equation901} to automatically satisfy the nonnegativity requirement, for any value of
$\Delta t$.

Turning to the approximation for the discrete regime, it is immediately clear from the form of \eqref{equation412} that if all of the
values $u_{C}^{n,i}$ and $u_{D}^{n,i}$ are nonnegative, then each $u_{D}^{n+1,i}$ will be nonnegative if for each $i=1,\dots,N$ we have
that $1-\Delta ta_i$ is nonnegative. This can be ensured by taking
\begin{equation}\label{equation415}
0<\Delta t \leq \frac{1}{a_i} \hspace{3mm} \text{for}\hspace{2mm}i=2,\dots,N \hspace{2mm} \text{such that} \hspace{2mm} a_i\neq0.
\end{equation}
Therefore if we choose a $\Delta t$ small enough that both \eqref{equation414} and \eqref{equation415} are satisfied, then our approximate
solutions will remain nonnegative.
\end{proof}

From now on we shall assume that conditions \eqref{equation414} and \eqref{equation415} are satisfied and that
$c_0(x)\geq 0$ and each $d_{0,i}\geq 0$ so that our approximations remain nonnegative.\\

\begin{remark}\label{remark401}The bound \eqref{equation414} is dependent on the mesh and it is perhaps not immediately apparent how this bounding value might vary as we refine the mesh. In particular, it would be advantageous to confirm that it is indeed possible to find a constant $k_1$, such that conditions \eqref{equation404} and \eqref{equation414} can be satisfied simultaneously, whilst $h\searrow0$. In the upcoming analysis we will place restrictions on the functions $a$ and $b$; these constraints will allow us to guarantee the existence of such a $k_1$.

The upcoming Theorem~\ref{theorem501} imposes the restriction $a,b\in L_{\infty}$ on the restricted domains $[N,R]$ and $[N,R]\times[N,R]$ respectively, with $\alpha(R)$ and $\beta(R)$ being the essential suprema for $a$ and $b$ on said domains. This being the case, we have $A_i\leq \alpha(R)$ and $B_{k,i}\leq\beta(R)$ for all values of $i$ and $k$ admissible in \eqref{equation414}. Furthermore, from \eqref{equation303} we may deduce that each $b_i(y)\leq y$, hence $\tilde{B}_{j,i}\leq R$. Finally, all mesh midpoints $x_i$ must clearly satisfy $x_i\geq N\geq1>h$. Taken together, these bounds lead, via a simple calculation, to
\begin{align*}
\frac{h}{\alpha(R)\left(\beta(R)R\left(R-N\right)+RN(N+1)/2\right)}
\leq \frac{x_i}{A_i\left(\sum_{k=0}^{i-1}x_{k}B_{k,i}\Delta x_{k}+\sum_{j=1}^{N}j\tilde{B}_{j,i}\right)},
\end{align*}
for $i=0,1,\dots,I_h-1$. Hence, we have established a possible value for $k_1$, which ensures \eqref{equation404} and \eqref{equation414} can be satisfied simultaneously as $h\searrow0$.
\end{remark}

\subsection{Mass Conservation by the Numerical Solutions}

In \cite[Lemma 6.2]{baird18}, the exact solution to our system of equations \eqref{equation301} and \eqref{equation302} was shown to conserve mass between the two regimes. We now show that this property is shared by our numerical solutions.

\begin{lemma}\label{lemma402} The approximate solutions generated by \eqref{equation409} and \eqref{equation412} conserve mass.
\end{lemma}
\begin{proof} The mass associated with the approximate continuous regime solution, $u_C^h(x,t)$, is given by
\begin{align}\label{equation416}
\|u^h_C(\cdot,t)\|_{L_{1}^1(N,R)}&=\int_N^R\sum_{n=0}^{M-1} \sum_{i=0}^{I_h-1} \chi_{\Lambda_i}(x)\chi_{\tau_n}(t)u_{C}^{n,i}\,x\dd x\nonumber\\
&=\sum_{n=0}^{M-1}\chi_{\tau_n}(t) \sum_{i=0}^{I_h-1} u_{C}^{n,i}\int_N^R \chi_{\Lambda_i}(x)\,x\dd x\nonumber\\
&=\sum_{n=0}^{M-1}\chi_{\tau_n}(t) \sum_{i=0}^{I_h-1}x_i \Delta x_i u_{C}^{n,i},
\end{align}
whilst the approximate solution $u_D^h(t)$ has associated mass given by

\[  \|u_{D}^h(t)\|_{X_D}=\sum_{i=1}^N i u_{Di}^h(t)= \sum_{i=1}^N i \sum_{n=0}^{M-1} \chi_{\tau_n}(t)u_{D}^{n,i} =\sum_{n=0}^{M-1}
\chi_{\tau_n}(t)  \sum_{i=1}^N i u_{D}^{n,i}.\]
Summing these two expressions gives the total mass:
\begin{equation}\label{equation417}
M^h(t)=\sum_{n=0}^{M-1}\chi_{\tau_n}(t)\left( \sum_{i=0}^{I_h-1}x_i \Delta x_i u_{C}^{n,i}+  \sum_{i=1}^N i u_{D}^{n,i}\right).
\end{equation}
First let us examine the mass accounted for by the continuous regime. From the relation \eqref{equation409} we get
\begin{align}\label{equation418}
&\sum_{i=0}^{I_h-1}x_iu_{C}^{n+1,i}\Delta x_{i}=\sum_{i=0}^{I_h-1}x_i\left(u_{C}^{n,i}+\frac{\Delta t}{x_i\Delta x_{i}}(F_{i+1/2}^n-F_{i-1/2}^n)-
\frac{\Delta t}{x_i}S_i^n\right)\Delta x_{i}\nonumber\\
&=\sum_{i=0}^{I_h-1}x_iu_{C}^{n,i}\Delta x_{i}+\Delta t\sum_{i=0}^{I_h-1}(F_{i+1/2}^n-F_{i-1/2}^n)-\Delta t\sum_{i=0}^{I_h-1}S_i^n\Delta x_{i}
\nonumber\\
&=\sum_{i=0}^{I_h-1}x_iu_{C}^{n,i}\Delta x_{i}-\Delta t\sum_{i=0}^{I_h-1}S_i^n\Delta x_{i}.
\end{align}
The middle summation term is lost in going to the final line as the sum is telescoping with zero end terms.
Now we consider the discrete regime mass; the generating relation \eqref{equation412} gives us
\begin{align}\label{equation419}
&\sum_{i=1}^{N}iu_{D}^{n+1,i}=\sum_{i=1}^{N}i\left((1-\Delta ta_i)u_{D}^{n,i}+\Delta t\sum_{j=i+1}^{N}a_jb_{i,j}u_{D}^{n,j}
+\Delta t\sum_{j=0}^{I_h-1}A_j
\tilde{B}_{i,j}u_{C}^{n,j}\Delta x_{j}\right)\nonumber\\
&=\sum_{i=1}^{N}iu_{D}^{n,i}-\Delta t\sum_{i=1}^{N}ia_iu_{D}^{n,i}+\Delta t\sum_{i=1}^{N}i\sum_{j=i+1}^{N}a_jb_{i,j}u_{D}^{n,j}+\Delta t
\sum_{i=1}^{N}i
\sum_{j=0}^{I_h-1}A_j\tilde{B}_{i,j}u_{C}^{n,j}\Delta x_{j}\nonumber\\
&=\sum_{i=1}^{N}iu_{D}^{n,i}-\Delta t\sum_{i=2}^{N}ia_iu_{D}^{n,i}
+\Delta t\sum_{j=2}^{N}a_ju_{D}^{n,j}\hspace{-1mm}\left(\sum_{i=1}^{j-1}ib_{i,j}\right)+\Delta t\hspace{-.2mm}\sum_{j=0}^{I_h-1}\Delta x_{j}\hspace{-.2mm}\left(\hspace{-.2mm}A_ju_{C}^{n,j}\sum_{i=1}^{N}i\tilde{B}_{i,j}
\hspace{-.2mm}\right)\nonumber\\
&=\sum_{i=1}^{N}iu_{D}^{n,i}+\Delta t\sum_{j=0}^{I_h-1}\Delta x_{j}S_j^n.
\end{align}
The middle two terms cancel due to the mass conservation condition \eqref{equation304}. Combining equations \eqref{equation418} and
\eqref{equation419} we obtain
\begin{align*}
&\sum_{i=0}^{I_h-1}x_iu_{C}^{n+1,i}\Delta x_{i}+\sum_{i=1}^{N}iu_{D}^{n+1,i}\\
&=\sum_{i=0}^{I_h-1}x_iu_{C}^{n,i}\Delta x_{i}-\Delta t\sum_{i=0}^{I_h-1}S_i^n\Delta x_{i}+\sum_{i=1}^{N}iu_{D}^{n,i}+\Delta t\sum_{j=0}^{I_h-1}
\Delta x_{j}S_j^n\\
&=\sum_{i=0}^{I_h-1}x_iu_{C}^{n,i}\Delta x_{i}+\sum_{i=1}^{N}iu_{D}^{n,i}.
\end{align*}
From repeated application of this equality it is easily seen that the bracketed expression appearing in \eqref{equation417} is equal for all values
of $n$, and hence the total mass $M^h(t)$ remains constant.
\end{proof}

\section{Convergence of the Scheme to a Weak Solution as  \texorpdfstring{$h\rightarrow 0$}{}}
Having determined the nonnegativity and mass conservative properties of the approximate solutions provided by \eqref{equation4105} and \eqref{equation4125}, in this section we set out to prove that they
converge, in some sense, to a limit as the parameter $h$, and by necessity $\Delta t$, go to zero, and show that this limit itself is an `exact' solution to our truncated model.

\subsection{Continuous Fragmentation Regime: Convergence}
Let us start with the continuous regime approximations $\left\lbrace u^h_C \right \rbrace$. In order to prove the (weak)
convergence of this family, we employ a weak compactness argument, utilising the Dunford--Pettis theorem (Theorem~\ref{theorem205}),
which provides necessary and sufficient conditions for weak compactness in an $L_{1}$ space. We begin by proving the equiboundedness of
the set $\left\lbrace u^h_C \right \rbrace$.

\begin{lemma}\label{lemma501}The family of approximations $\left\lbrace u_C^h\right\rbrace$ is equibounded (uniformly bounded) in the space $L_1((N,R) \times [0,T),x \dd x \dd t)$.
\end{lemma}

\begin{proof}
Recalling equation \eqref{equation416}, we have for any $t\in [0,T)$ that
\begin{align*}
\|u^h_C(\cdot,t)\|_{L_{1}^1(N,R)}&=\sum_{n=0}^{M-1}\chi_{\tau_n}(t) \sum_{i=0}^{I_h-1}x_i \Delta x_i u_{C}^{n,i}.
\end{align*}
From the analysis of Lemma~\ref{lemma401}, each of the values $u_{C}^{n,i}$ is nonnegative, and as such the values $S_i^n$ are nonnegative. Therefore, from the last line of equation \eqref{equation418} we deduce that
\[\sum_{i=0}^{I_h-1}x_i \Delta x_i u_{C}^{n,i}\leq \sum_{i=0}^{I_h-1}x_i \Delta x_iu_{C}^{n-1,i}\hspace{3mm}\text{for}\hspace{2mm}n=1,\ldots,M-1.\]

Repeated application of this inequality yields
\begin{equation}\label{equation501}
\sum_{i=0}^{I_h-1}x_i \Delta x_i u_{C}^{n,i}\leq \sum_{i=0}^{I_h-1}x_i \Delta x_i u_{C}^{0,i}=\sum_{i=0}^{I_h-1}x_i \int_ {\Lambda_{i}} c_0(x)\dd x\leq\sum_{i=0}^{I_h-1}\frac{x_i}{x_{i-1/2}} \int_ {\Lambda_{i}} c_0(x)\,x\dd x.
\end{equation}
The quantity $\frac{x_i}{x_{i-1/2}}$ can be bounded as follows:
\begin{equation*}
\frac{x_i}{x_{i-1/2}}=\frac{x_{i-1/2}+\frac{1}{2}\Delta x_i}{x_{i-1/2}}\leq 1+ \frac{h}{2N}\leq \frac{3}{2}.
\end{equation*}
We note this bound as it will appear regularly in subsequent calculations. Substituting this within \eqref{equation501} yields
\[\sum_{i=0}^{I_h-1}x_i \Delta x_i u_{C}^{n,i}\leq\frac{3}{2}\|c_0\|_{L_{1}^1(N,R)},\]
for $n=0,\ldots,M-1$. Replacing this inequality in our calculation gives us the following, which holds for all $t\in [0,T)$:
\begin{align*}
\|u^h_C(\cdot,t)\|_{L_{1}^1(N,R)}&\leq\sum_{n=0}^{M-1}\chi_{\tau_n}(t)\frac{3}{2}\|c_0\|_{L_{1}^1(N,R)}= \frac{3}{2}\|c_0\|_{L_{1}^1(N,R)}.
\end{align*}
Integrating this inequality with respect to $t$ from $0$ to $T$ we obtain the required equiboundedness of $\left\lbrace u_C^h\right\rbrace$ in the space $L_1((N,R) \times [0,T),x\dd x\dd t)$.
\end{proof}

We now move on to prove the second of the two required conditions for the Dunford--Pettis theorem, namely equiintegrability. However, prior to this we collect all the conditions so far imposed on our model via the functions $a$, $b$ and $b_i$, and the initial distributions $c_0$ and $d_0$, and also on our mesh via the parameters $h$ and $\Delta t$.

\begin{remark}\label{conditions}Throughout the remaining analysis, the following conditions shall be assumed to be satisfied.
\begin{enumerate}
\item The initial mass distributions within the continuous and discrete regimes are nonnegative, that is $c_0(x)$ for all $x>N$ and $d_{0,i}\geq0$ for $i=1,\ldots,N$.
\item The functions $a$ and $b$ are nonnegative and belong to the spaces $L_{\infty,loc}\left([N,\infty)\right)$ and \\$L_{\infty,loc}\left([N,\infty)\times[N,\infty)\right)$, respectively.
\item Each of the functions $b_i$ is assumed to be nonnegative. This nonnegativity in conjunction with condition $2$ is sufficient to guarantee that $b_i\in L_{\infty,loc}\left([N,\infty)\right)$, as per Remark \ref{convremark1}.
\item There exist positive constants $k_1$ and $k_2$ such that the mesh parameters $h$ and $\Delta t$ satisfy
\[k_1h\leq \Delta t \leq k_2 h.\]
\item To ensure that the approximate solutions remain nonnegative, the time step $\Delta t$ is assumed to satify the following constraints:
\[0<\Delta t \leq \frac{x_i}{A_i\left(\sum_{k=0}^{i-1}x_{k}B_{k,i}\Delta x_{k}+\sum_{j=1}^{N}j\tilde{B}_{j,i}\right)}\hspace{3mm}\text{for}
\hspace{2mm}i=0,1,\dots,I_h-1,\]
for all cases of the denominator being nonzero, and
\[0<\Delta t \leq \frac{1}{a_i} \hspace{3mm} \text{for}\hspace{2mm}i=2,\dots,N \hspace{2mm} \text{such that} \hspace{2mm} a_i\neq0.\]
\item There exists a constant $\theta>0$ such that
\[K(R) \Delta t \leq \theta<1, \]
where $K(R)=\alpha(R)\beta(R) R$, with $\alpha(R)$ and $\beta(R)$ being the essential suprema for $a$ and $b$ on the restricted domains $[N,R]$ and $[N,R]\times[N,R]$, respectively.
\end{enumerate}
\end{remark}

\begin{theorem}\label{theorem501} Under the assumptions outlined in Remark \ref{conditions}, the family $\left\lbrace u_C^h\right\rbrace$ is equiintegrable in $L_1((N,R) \times[0,T),x\dd x \dd t)$.
\end{theorem}

\begin{proof}
Consider the constant sequence comprising solely of the initial data $c_0\in L_{1}\left((N,R),x\dd x\right)$. Clearly this sequence is
convergent, therefore $\left\lbrace c_0\right\rbrace$ forms a weakly sequentially compact set in $L_{1}\left((N,R),x\dd x\right)$. Hence
by the de la Vallee Poussin theorem (Theorem~\ref{theorem204}) there exists a nonnegative, convex function $\Phi \in C^\infty ([0,\infty))$, with $\Phi(0)=0$ and $\Phi'(0)=1$ such that $\Phi'$ is concave and satisfies
\begin{equation*}
\frac{\Phi(x)}{x}\rightarrow\infty\hspace{2mm}\text{as}\hspace{2mm}x\rightarrow\infty \hspace{3mm}\text{and}\hspace{3mm}\int_{N}^R\Phi(c_0)(x)\,x\dd x
<\infty.
\end{equation*}
A standard inequality \eqref{equation203}, for $C^1$ convex functions gives us
\begin{equation*}
\Phi(u_{C}^{n+1,i})-\Phi(u_{C}^{n,i})\leq\left(u_{C}^{n+1,i}-u_{C}^{n,i}\right)\Phi'(u_{C}^{n+1,i}).
\end{equation*}
Multiplying this by $x_i \Delta x_i$ and summing over all $i$ gives
\begin{equation*}
\sum_{i=0}^{I_h-1}x_i\Delta x_i \left(\Phi\left(u_{C}^{n+1,i}\right)-\Phi\left(u_{C}^{n,i}\right) \right)
\leq \sum_{i=0}^{I_h-1}x_i \Delta x_i\left(\left(u_{C}^{n+1,i}-u_{C}^{n,i}\right)\Phi'(u_{C}^{n+1,i})\right).
\end{equation*}
Utilising equation~\eqref{equation409} we can rewrite this as
\begin{align}\label{equation504}
&\sum_{i=0}^{I_h-1}x_i\Delta x_i \left(\Phi\left(u_{C}^{n+1,i}\right)-\Phi\left(u_{C}^{n,i}\right)  \right)\nonumber\\
&\leq \sum_{i=0}^{I_h-1}x_i\Delta x_i \left(\frac{\Delta t}{x_i\Delta x_{i}}(F_{i+1/2}^n-F_{i-1/2}^n)-\frac{\Delta t}{x_i}S_i^n\right)
\Phi'\left(u_{C}^{n+1,i}\right).
\end{align}

Recalling the definition of $S_i^n$ from \eqref{sink}, we see that it must be nonnegative. Additionally, Lemma~\ref{lemma203}(ii) and Lemma~\ref{lemma401} give $\Phi'\left(u_{C}^{n+1,i}\right)\geq0$, hence we can drop the term involving $S_i^n$ from \eqref{equation504} and the inequality will still remain valid, giving us
\begin{equation*}
\sum_{i=0}^{I_h-1}x_i\Delta x_i \left(\Phi\left(u_{C}^{n+1,i}\right)-\Phi\left(u_{C}^{n,i}\right)  \right)
\leq \sum_{i=0}^{I_h-1} \Delta t(F_{i+1/2}^n-F_{i-1/2}^n)\Phi'\left(u_{C}^{n+1,i}\right).
\end{equation*}
With some easy modification, equation~\eqref{equation413} becomes the inequality
\[F_{i+1/2}^n-F_{i-1/2}^n\leq x_{i}\Delta x_{i}\sum_{j=i+1}^{I_h-1}A_jB_{i,j}u_{C}^{n,j}\Delta x_{j},\]
which, if placed in the previous inequality, results in

\begin{equation*}
\sum_{i=0}^{I_h-1}x_i \Delta x_i \left(\Phi\left(u_{C}^{n+1,i}\right)-\Phi\left(u_{C}^{n,i}\right)  \right)
\leq \Delta t\sum_{i=0}^{I_h-1}\sum_{j=i+1}^{I_h-1}x_i A_jB_{i,j}u_{C}^{n,j}\Delta x_{i}\Delta x_{j}\Phi'\left(u_{C}^{n+1,i}\right).
\end{equation*}

Utilising Lemma~\ref{lemma203}(i) with $x=u_{C}^{n,j}$ and $y=u_{C}^{n+1,i}$ and noting that the constants $\alpha(R)$ and $\beta(R)$ bound the average values $A_j$ and $B_{i,j}$, we get

\begin{align*}
&\sum_{i=0}^{I_h-1 }x_i\Delta x_i \left( \Phi\left(u_{C}^{n+1,i}\right)-\Phi\left(u_{C}^{n,i}\right)  \right)
\leq \alpha \beta \Delta t\sum_{i=0}^{I_h-1} \sum_{j=i+1}^{I_h-1}x_i\Delta x_{i} \Delta x_{j} u_{C}^{n,j}
\Phi'\left(u_{C}^{n+1,i}\right)\\
&\leq \alpha(R) \beta(R) \Delta t \left(\sum_{i=0}^{I_h-1} \left(x_i \Delta x_{i} \Phi\left(u_{C}^{n+1,i}\right)\sum_{j=i+1}^{I_h-1}\Delta x_j\right)+\sum_{i=0}^{I_h-1}\left( x_i \Delta x_{i}  \sum_{j=i+1}^{I_h-1} \Delta x_j \Phi(u_{C}^{n,j})\right)\right).
\end{align*}
As $j$ is restricted to be greater than $i$ we have $x_j>x_i$ for admissible $j$ and $i$. This allows us to switch $x_i$ for $x_j$ in the
second term and take this within the inner summation. Following this we expand the summation over $j$ to give
\begin{align*}
&\sum_{i=0}^{I_h-1 }x_i\Delta x_i \left( \Phi\left(u_{C}^{n+1,i}\right)-\Phi\left(u_{C}^{n,i}\right)  \right)\\
&\leq \alpha(R) \beta(R) \Delta t \left(\sum_{i=0}^{I_h-1} \left(x_i \Delta x_{i} \Phi\left(u_{C}^{n+1,i}\right)\sum_{j=i+1}^{I_h-1}\Delta x_j\right)+\sum_{i=0}^{I_h-1}\left(\Delta x_{i}  \sum_{j=i+1}^{I_h-1}x_j \Delta x_j \Phi(u_{C}^{n,j})\right)\right)\\
&\leq \underbrace{\alpha(R)\beta(R) R}_{=K(R)} \Delta t \left(\sum_{i=0}^{I_h-1} x_i \Delta x_{i} \Phi\left(u_{C}^{n+1,i}\right)
+ \sum_{j=0}^{I_h-1}x_j \Delta x_j \Phi(u_{C}^{n,j})\right).
\end{align*}

If we change the index variable from $j$ to $i$ in the second summation and re-arrange then we obtain
\begin{align*}
\left(1-K(R)\Delta t\right)\sum_{i=0}^{I_h-1} x_i\Delta x_{i}\Phi(u_{C}^{n+1,i})\leq \left(1+K(R)\Delta t\right)\sum_{i=0}^{I_h-1}x_i \Delta x_{i}
\Phi\left(u_{C}^{n,i}\right).
\end{align*}
Some further manipulations produce
\begin{align*}
\left(1-K(R)\Delta t\right)\sum_{i=0}^{I_h-1}x_i \Delta x_{i}\left(\Phi(u_{C}^{n+1,i})-\Phi(u_{C}^{n,i})\right)
\leq 2K(R)\Delta t\sum_{i=0}^{I_h-1}x_i \Delta x_{i}\Phi\left(u_{C}^{n,i}\right).
\end{align*}
By the final assumption of Remark \ref{conditions}, we have $1-K(R)\Delta t>0$ allowing us to divide through to get
\begin{align*}
\sum_{i=0}^{I_h-1}x_i \Delta x_{i}\Phi(u_{C}^{n+1,i})\leq\left(1+\frac{2K(R)\Delta t}{1-K(R)\Delta t}\right)\sum_{i=0}^{I_h-1}x_i \Delta x_{i}
\Phi\left(u_{C}^{n,i}\right).
\end{align*}
Repeated application of this inequality yields
\begin{align*}
\sum_{i=0}^{I_h-1}x_i \Delta x_{i}\Phi(u_{C}^{n+1,i})&\leq\left(1+\frac{2K(R)\Delta t}{1-K(R)\Delta t}\right)^{n+1}\hspace{1mm}\sum_{i=0}^{I_h-1} x_i
\Delta x_{i}\Phi\left(u_{C}^{0,i}\right)\\
&\leq\exp\left(\frac{2K(R)\Delta t (n+1)}{1-K(R)\Delta t}\right)\sum_{i=0}^{I_h-1}x_i \Delta x_{i}\Phi\left(u_{C}^{0,i}\right).
\end{align*}
For values of $t$ in the interval $\tau_n=[t_n,t_{n+1})$ this gives us
\begin{align*}
\int_{N}^{R}\Phi \left(u^h_C(x,t)\right)\,x\dd x
&=\sum_{i=0}^{I_h-1} x_i\Delta x_{i}\Phi\left(u_{C}^{n,i}\right)\\
&\leq\exp\left(\frac{2K(R)t}{1-K(R)\Delta t} \right)\sum_{i=0}^{I_h-1}x_i \Delta x_{i}\Phi\left(u_{C}^{0,i}\right)\\
&=\exp\left(\frac{2K(R)t}{1-K(R)\Delta t} \right)\sum_{i=0}^{I_h-1} x_i \Delta x_{i}\Phi\left(\frac{1}{\Delta x_{i}}\int_ {\Lambda_{i}} c_0(x)\dd x
\right).
\end{align*}
An application of Jensen's inequality \cite[Theorem 2.2]{lieb01} allows us to switch the order of $\Phi$ and integration to get
\begin{align*}
\int_{N}^{R}\Phi \left(u^h_C(x,t)\right)\,x\dd x
&\leq\exp\left(\frac{2K(R)t}{1-K(R)\Delta t} \right)\sum_{i=0}^{I_h-1} x_i \int_ {\Lambda_{i}} \Phi( c_0(x))\dd x\\
&\leq\frac{3}{2}\exp\left(\frac{2K(R)t}{1-K(R)\Delta t} \right)\sum_{i=0}^{I_h-1} \int_ {\Lambda_{i}} \Phi( c_0(x))\,x\dd x.
\end{align*}
By assumption $6$ of Remark \ref{conditions}, that $K(R)\Delta t \leq \theta<1$, we deduce that
\begin{align*}
\int_{N}^{R} \Phi \left(u^h_C(x,t)\right)\,x\dd x
\leq\frac{3}{2}\exp\left(\frac{2K(R)t}{1-\theta}\right)\int_{N}^{R}\Phi( c_0(x))\,x\dd x,
\end{align*}
which holds for all $t\in[0,T)$. Integrating the inequality with respect to $t$ from $0$ to $T$ confirms the equiintegrability of the family $\left\lbrace u_C^h\right\rbrace$ in the space $L_1((N,R) \times [0,T),x\dd x\dd t)$.
\end{proof}

By Theorem~\ref{theorem205} (Dunford--Pettis theorem), the sequence $\left\lbrace u_C^h\right\rbrace$ forms a weakly sequentially compact set in the space $L_{1}\left((N,R)\times[0,T),x\dd x\dd t\right)$. This implies the existence of a subsequence $\left\lbrace u_C^{hj}\right\rbrace$ and a function $u_C^R\in L_{1}\left((N,R)\times[0,T),x\dd x\dd t\right)$ such that $u_C^{hj}\rightharpoonup u_C^R$ in $L_{1}\left((N,R)\times[0,T),x\dd x\dd t\right)$ as $j\rightarrow \infty$ and $h_{j}\rightarrow0$.\\

\begin{remark}\label{remark501} From now on this convergent subsequence will be considered implicitly, unless otherwise stated; as such we now use the notation $\left\lbrace u_C^h\right\rbrace$ to denote such a convergent subsequence, the choice of which, we note, may not be unique.
\end{remark}

\subsection{Continuous Fragmentation Regime: Weak Solution}

Having shown that our sequence of approximations converges (weakly) to a limit, we now aim to show that this limit provides a solution to our truncated equation \eqref{equation402}. Precisely, we intend to show that the function $u_C^R$ satisfies the following criterion.

\begin{definition}\label{definition501} The function $u_C^R$ is a weak solution of equation \eqref{equation402}, if it satisfies
\begin{align}\label{equation505}
&\int_{0}^{T}\int_{N}^{R} xu_C^R(x,t)\frac{\partial\varphi}{\partial t}(x,t) \dd x\dd t+\int_{N}^{R}xc_{0}(x)\varphi(x,0)\dd x\nonumber\\
&=\int_{0}^{T}\int_{N}^{R}\mathcal{F}^R\left(xu_{C}^R\right)\hspace{-.6mm}(x,t)\hspace{1mm}\frac{\partial\varphi}{\partial x}(x,t) \dd x\dd t
+\int_{0}^{T}\int_{N}^{R}S(xu_C^R)(x,t)\varphi(x,t)\dd x\dd t,
\end{align}
for all functions $\varphi$, which are twice continuously differentiable functions on $(N,R)\times(0,T)$, such that $\varphi$ and each of its derivatives up to order $2$ may be continuously extended to $[N,R]\times[0,T)$, and such that for each fixed $x\in[N,R]$, the support of $\varphi$ with respect to $t$ is a compact subset of $[0,T)$. We denote the set of such extended functions by $C_c^2\left([N,R]\times[0,T)\right)$. Finally, we note that the weak formulation \eqref{equation505} was obtained from \eqref{equation402} in the usual manner, recalling the zero boundary conditions established in Lemma~\ref{lastthing}.
\end{definition}

\begin{remark}\label{remark502} We now make note of a property of the function $\varphi$ and its derivatives, which we will make use of in our analysis. As $\varphi$ has compact support and is identically zero outwith this support, its derivatives, both first and second, must also be zero outwith the support. Now within this compact support, $\varphi$ and its derivatives are continuous and so must be bounded functions.
\end{remark}

\begin{definition}\label{definition502}
In the analysis which follows we make use of the following three approximations to $x$ over the domain $(N,R)$. First we have the left endpoint approximation, defined by
\[\xi^h:x \in (N,R)\rightarrow \xi^h(x)= \sum_{i=0}^{I_h-1} \chi_{\Lambda_i}(x)x_{i-1/2}.\]
Secondly we consider the midpoint approximation, defined by
\[X^h:x \in (N,R)\rightarrow X^h(x)= \sum_{i=0}^{I_h-1} \chi_{\Lambda_i}(x)x_i, \]
and finally we introduce the right endpoint approximation given by
\[\Xi^h:x \in (N,R)\rightarrow \Xi^h(x)= \sum_{i=0}^{I_h-1} \chi_{\Lambda_i}(x)x_{i+1/2}. \]
\end{definition}
\begin{remark}\label{convremark3}
It is a simple exercise to show that the three approximations, introduced above, converge point-wise (uniformly) to $x$ over the domain $(N,R)$ as the mesh parameter $h$ goes to $0$. The reader may find details given in \cite[Lemma 5.2.3]{bairdphd}.
\end{remark}

We are now in a position to proceed with our proof that $u_C^R$ is a weak solution to \eqref{equation402}.

\begin{definition}\label{definition503} Let $\varphi\in C_c^2\left([N,R]\times[0,T)\right)$ , then for sufficiently small $\Delta t$, the support of $\varphi$ with respect to $t$ lies within $[0,t_{M-1}]$. We define $\varphi_{i}^{n}$ as an approximation of $\varphi$ on $\Lambda_{i}\times \tau_n$ by
\[\varphi_i^n=\frac{1}{\Delta t}\int_{\tau_n}\varphi(x_{i-1/2},t)\dd t,\]
with $\varphi_{i}^{M-1}=\varphi_{i}^{M}=0$ for admissible $i$ and define $\varphi_{I_h}^{n}=0$ for all $n$.
\end{definition}

Rearranging equation \eqref{equation409}, multiplying by $\varphi_{i}^{n}$ and summing over $n=0,\ldots,M-1$ and $i=0,\ldots,I_h-1$, gives us
\begin{equation*}
\sum_{n=0}^{M-1} \sum_{i=0}^{I_h-1}\left(x_i\Delta x_{i}\left(u_{C}^{n+1,i}-u_{C}^{n,i}\right)\varphi_{i}^{n}-\Delta t(F_{i+1/2}^n-F_{i-1/2}^n)
\varphi_{i}^{n}+\Delta t\Delta x_{i}S_i^n\varphi_{i}^{n}\right)=0.
\end{equation*}
Rearrangement of the summations and utilising the compact support of $\varphi$ and the zero boundary flux gives us the following equality:
\begin{align}\label{equation506}
&\sum_{n=0}^{M-1} \sum_{i=0}^{I_h-1}x_i\Delta x_{i}u_{C}^{n+1,i}\left(\varphi_{i}^{n+1}-\varphi_{i}^{n}\right)
+ \sum_{i=0}^{I_h-1}x_i \Delta x_{i} u_{C}^{0,i}\varphi_{i}^{0}\nonumber\\
&-\sum_{n=0}^{M-1} \sum_{i=0}^{I_h-1}\Delta t F_{i+1/2}^n(\varphi_{i+1}^{n}-\varphi_{i}^{n})
-\sum_{n=0}^{M-1} \sum_{i=0}^{I_h-1}\Delta t\Delta x_{i}S_i^n\varphi_{i}^{n}
=0.
\end{align}
The above equality can be seen as the discrete equivalent of the weak formulation \eqref{equation505}. Our approach now involves taking the limit as $h\rightarrow 0$ of \eqref{equation506} and showing that we do indeed obtain \eqref{equation505} with $u_C^R$ as a weak solution. Observing the terms of \eqref{equation505} we see that the integrals are with respect to the measure $\ddn x\dd t$ whilst we have shown that convergence occurs in the space with weighted measure $x \dd x \dd t$. At this point we highlight the use of Theorem~\ref{theorem201} to switch spaces but retain convergence.

\begin{theorem}\label{theorem502}
Under the assumptions outlined in Remark \ref{conditions}, the function $u_C^R$ obtained as the limit of the sequence $\left\lbrace u_C^h \right\rbrace$, is a weak solution of our equation, satisfying \eqref{equation505}.
\end{theorem}

\begin{proof}
Looking initially at the first two terms of \eqref{equation506}, we can express them as follows:
\begin{align*}
&\sum_{n=0}^{M-1} \sum_{i=0}^{I_h-1}x_i\Delta x_{i}u_{C}^{n+1,i}\left(\varphi_{i}^{n+1}-\varphi_{i}^{n}\right)
+ \sum_{i=0}^{I_h-1}x_i \Delta x_{i} u_{C}^{0,i}\varphi_{i}^{0}\nonumber\\
&=\underbrace{\sum_{i=0}^{I_h-1}x_i\Delta x_{i}u_{C}^{M,i}\left(\varphi_{i}^{M}-\varphi_{i}^{M-1}\right)}_{=0}\\
&\hspace{5mm}+\sum_{n=0}^{M-2} \sum_{i=0}^{I_h-1}\int_{\tau_{n+1}} \int_{\Lambda_{i}}X^h(x)u^h_C(x,t)
\frac{\varphi( \xi^h(x),t)-\varphi( \xi^h(x),t-\Delta t)}{\Delta t}\dd x \dd t\\
&\hspace{5mm}+\sum_{i=0}^{I_h-1}\int_{\Lambda_{i}}X^h(x)u^h_C(x,0)\frac{1}{\Delta t}\int_{0}^{\Delta t}\varphi( \xi^h(x),t)\dd t \dd x\\
&=\int_{0}^{T} \int_{N}^{R} \chi_{[\Delta t,T)}(t) X^h(x)u^h_C(x,t)
\frac{\varphi( \xi^h(x),t)-\varphi( \xi^h(x),t-\Delta t)}{\Delta t}\dd x \dd t\\
&\hspace{5mm}+\int_{N}^{R}X^h(x)u^h_C(x,0)\frac{1}{\Delta t}\int_{0}^{\Delta t}\varphi( \xi^h(x),t)\dd t\dd x.
\end{align*}

Considering the first of the double integrals, let $(x,t)\in(N,R)\times(0,T)$, then in the case that $0<\Delta t \leq t$, Taylor expansions of the $\varphi$ terms about the point $(x,t)$ give us
\[\varphi( \xi^h(x),t)=\varphi(x,t)+(\xi^h(x)-x)\frac{\partial \varphi}{\partial x}(x,t)+O(h^2),\]
\[\varphi( \xi^h(x),t-\Delta t)=\varphi(x,t)+(\xi^h(x)-x)\frac{\partial \varphi}{\partial x}(x,t)
+(t-\Delta t-t)\frac{\partial \varphi}{\partial t}(x,t)+O(h^2,h\Delta t,\Delta t^2).\]
Simple cancellations and recalling the condition \eqref{equation404} relating $h$ and $\Delta t$ give us
\begin{equation}\label{equation507}
\chi_{[\Delta t,T)}(t)\frac{\varphi( \xi^h(x),t)-\varphi( \xi^h(x),t-\Delta t)}{\Delta t}=\left\lbrace \begin{array}{cc}&\hspace{-5mm}\dfrac{\Delta t\frac{\partial \varphi}{\partial t}(x,t)+O(\Delta t^2)}{\Delta t}, \hspace{6mm}\Delta t \leq t\\ \vspace{-3mm}\\&\hspace{10.5mm}0,\hspace{21.2mm}\Delta t > t\end{array}\right..
\end{equation}
The expression on the left-hand side of \eqref{equation507} can thus be seen to converge pointwise to $\varphi_t$ on $(N,R)\times(0,T)$ as the mesh size goes to 0. Furthermore, since the derivatives of $\varphi$ are bounded as per Remark \ref{remark502}, we can bound the left-hand side of \eqref{equation507} on $(N,R)\times(0,T)$, with a bound that is uniform w.r.t $h$, as $h\searrow 0$. As noted in Remark \ref{convremark3}, the functions $X^h(x)$ converge pointwise to $x$ on $(N,R)$ as $h\searrow 0$, and are clearly bounded by $R$ for all values of $x$ and $h$. Noting that $(N,R)\times[0,T) \backslash (N,R)\times(0,T)$ is of measure $0$ with respect to the measure $\dd x \dd t$, we see that the terms accompanying $u_C^h$ in our double integral satisfy the conditions for $\left\lbrace g_h \right\rbrace$ from Theorem~\ref{theorem203}. We have shown previously that $u_C^h\rightharpoonup u_C^R$ in $L_{1}\left((N,R)\times[0,T),\,x\dd x\dd t\right)$ and by Theorem~\ref{theorem201} this is also the case in $L_{1}\left((N,R)\times[0,T),\dd x \dd t\right)$ and so an application of Theorem~\ref{theorem203} gives us
\begin{align}\label{equation508}
\int_{0}^{T} \int_{N}^{R} \chi_{[\Delta t,T)}(t) & X^h(x)u^h_C(x,t)
\frac{\varphi( \xi^h(x),t)-\varphi( \xi^h(x),t-\Delta t)}{\Delta t} \dd x \dd t\rightarrow \nonumber\\
&\int_{0}^{T}\int_{N}^{R} xu_C^R(x,t)\frac{\partial\varphi}{\partial t}(x,t) \dd x \dd t,
\end{align}
as the mesh size parameter $h$ goes to 0. Next, we consider the second term appearing above. Since
$\varphi$ is $C_c^2\left([N,R]\times[0,T)\right)$, its derivatives are bounded, allowing us to deduce that
\begin{equation}\label{equation54322}
\frac{1}{\Delta t} \int_{0}^{\Delta t} \varphi(\xi^h(x),t)\dd t\rightarrow\varphi(x,0)\hspace{2mm}\text{as}\hspace{2mm}h\searrow0,
\end{equation}
for all $x\in (N,R)$, as we now demonstrate. Consider the following:
\begin{align}\label{equation54321}
&\left|\frac{1}{\Delta t}\int_{0}^{\Delta t}\varphi( \xi^h(x),t)\dd t-\varphi(x,0)\right|=
\left|\frac{1}{\Delta t}\int_{0}^{\Delta t}\left(\varphi( \xi^h(x),t)-\varphi(x,0)\right)\dd t\right|\nonumber\\
&=\left|\frac{1}{\Delta t}\int_{0}^{\Delta t}\left(\varphi( \xi^h(x),t)-\varphi(x,t)+\varphi(x,t)-\varphi(x,0)\right)\dd t\right|\nonumber\\
&\leq\frac{1}{\Delta t}\int_{0}^{\Delta t}\left|\varphi( \xi^h(x),t)-\varphi(x,t)\right|\dd t
+\frac{1}{\Delta t}\int_{0}^{\Delta t}\left|\varphi(x,t)-\varphi(x,0)\right|\dd t.
\end{align}
Expressing $\varphi( \xi^h(x),t)$ using a Taylor expansion about $(x,t)\in(N,R)\times(0,T)$, and recalling Remark~\ref{remark502} about the derivatives of $\varphi$ we get
\[\varphi( \xi^h(x),t)=\varphi(x,t)+(\xi^h(x)-x)\frac{\partial\varphi}{\partial x}(x,t)+O(h^2).\]
Hence bounding the derivative $\partial\varphi/\partial x$ and noting that $\left|\xi^h(x)-x\right|\leq h$ gives us
\[\left|\varphi( \xi^h(x),t)-\varphi(x,t)\right|\leq C_1h,\]
for some constant $C_1$ independent of $h$ and $\Delta t$. Similarly, expanding $\varphi(x,t)$ about $(x,0)$, where $x\in(N,R)$, produces
\begin{equation}\label{Tseries}
\varphi( x,t)=\varphi(x,0)+\underbrace{(t-0)}_{\leq \Delta t}\frac{\partial\varphi}{\partial t}(x,0_+)+O(h^2),
\end{equation}
for $t\geq0$. The use of the notation $\tfrac{\partial\varphi}{\partial t}(x,0_+)$ signifies we are considering the right derivative of $\varphi$ with respect to $t$ at $t=0$.
The expansion \eqref{Tseries} then leads to
\begin{align*}
\left|\varphi(x,t)-\varphi(x,0)\right|\leq C_2h,
\end{align*}
for some other constant $C_2$, independent of $h$ and $\Delta t$. Returning to \eqref{equation54321} we have
\begin{align*}
\left|\frac{1}{\Delta t}\int_{0}^{\Delta t}\varphi( \xi^h(x),t)\dd t-\varphi(x,0)\right|\leq (C_1+C_2)h.
\end{align*}
Hence \eqref{equation54322} does indeed hold  for $x\in (N,R)$, furthermore the convergence is uniform with respect to $x$. Together with the pointwise convergence of $X^h(x)$ to $x$, and from  Remark~\ref{convremark2}, the $L_1(N,R)$ strong convergence  of $u^h_C(x,0)$  to the restriction of $c_0$ to $(N,R)$, another application of Theorem~\ref{theorem203} yields
\begin{equation}\label{equation509}
\int_{N}^{R}X^h(x)u^h_C(x,0)\frac{1}{\Delta t}\int_{0}^{\Delta t}\varphi( \xi^h(x),t)\dd t\dd x\rightarrow
\int_{N}^{R}xc_{0}(x)\varphi(x,0)\dd x.
\end{equation}

Moving on to the third term of equation \eqref{equation506}, for $t\in\tau_n$ and $x\in\Lambda_i$ we can write the numerical flux as an integral as follows:
\begin{align*}
F_{i+1/2}^{n}&=\sum_{j=i+1}^{I_h-1}\sum_{k=0}^{i}x_{k}A_jB_{k,j}u_{C}^{n,j}\Delta x_{k}\Delta x_{j}\\
             &=\sum_{j=i+1}^{I_h-1}\sum_{k=0}^{i}\int_{\Lambda_j}\int_{\Lambda_k}X^h(w)a^h(v)b^h(w|v) u_C^h(v,t)\dd w\dd v\\
             &=\int_{x_{i+1/2}}^{R}\int_{N}^{x_{i+1/2}}X^h(w)a^h(v)b^h(w|v) u_C^h(v,t)\dd w\dd v\\
             &=\int_{N}^{R}\int_{N}^{R} \chi_{[\Xi^h(x),R]}(v)\chi_{[N,\Xi^h(x)]}(w) X^h(w)a^h(v)b^h(w|v) u_C^h(v,t)\dd w\dd v\\
             &=:\mathcal{F}^h(u_C^h)(x,t).
\end{align*}
Then the third term of equation \eqref{equation506} is given by
\begin{align*}
&\sum_{n=0}^{M-1} \sum_{i=0}^{I_h-1}\Delta t F_{i+1/2}^n(\varphi_{i+1}^{n}-\varphi_{i}^{n})\\
&=\sum_{n=0}^{M-1} \sum_{i=0}^{I_h-1}F_{i+1/2}^n\int_{\tau_n}\varphi(x_{i+1/2},t)-\varphi(x_{i-1/2},t)\dd t\\
&=\sum_{n=0}^{M-1} \sum_{i=0}^{I_h-1}\int_{\tau_n}\int_{\Lambda_i}F_{i+1/2}^n\frac{\partial\varphi }{\partial x}(x,t)\dd x\dd t\\
&=\int_0^T \int_{N}^R\mathcal{F}^h(u_C^h)(x,t)\frac{\partial\varphi }{\partial x}(x,t)\dd x \dd t.
\end{align*}
Expressed in full this gives us the following, after a switch in the order of integration:
\begin{align}\label{equation510}
&\int_0^T\int_{N}^R\int_{N}^{R}\int_{N}^{R} \chi_{[\Xi^h(x),R]}(v)\chi_{[N,\Xi^h(x)]}(w) X^h(w)a^h(v)b^h(w|v) u_C^h(v,t)\dd w \dd v\frac{\partial\varphi }{\partial x}(x,t)\dd x\dd t\nonumber\\
&=\int_{N}^R\int_{N}^{R}\chi_{[N,\Xi^h(x)]}(w) X^h(w)\left(\int_0^T\int_{N}^{R} \chi_{[\Xi^h(x),R]}(v)a^h(v)b^h(w|v) u_C^h(v,t)\frac{\partial\varphi }{\partial x}(x,t)\dd v\dd t\right)\dd w\dd x.
\end{align}
Due to the boundedness of the partial derivative $\varphi_x$ and the $L_{\infty,loc}$ property of the functions $a$ and $b$, for almost all fixed $(x,w) \in (N,R) \times (N,R)$, the product $\chi_{[\Xi^h(x),R]}(v)a^h(v)b^h(w|v)\frac{\partial\varphi }{\partial x}(x,t)$ is a bounded (uniformly w.r.t. $h$) function of $v$ and $t$. Also, as a consequence of Remarks~\ref{convremark1} and~\ref{convremark3}, it converges pointwise almost everywhere (w.r.t the measure $\dd v\dd t$) on $(v,t) \in (N,R) \times [0,T)$ to $\chi_{[x,R]}(v)a(v)b(w|v)\frac{\partial\varphi }{\partial x}(x,t)$ as $h\rightarrow 0$. Since $u_C^h\rightharpoonup u_C^R$ in $L_1((N,R)\times[0,T),\dd x\dd t)$, an application of Theorem~\ref{theorem203} gives us
\begin{align*}
&\int_0^T\int_{N}^{R} \chi_{[\Xi^h(x),R]}(v)a^h(v)b^h(w|v) u_C^h(v,t)\frac{\partial\varphi }{\partial x}(x,t)\dd v\dd t\\
&\hspace{10mm}\rightarrow \int_0^T\int_{N}^{R} \chi_{[x,R]}(v)a(v)b(w|v) u_C^R(v,t)\frac{\partial\varphi }{\partial x}(x,t)\dd v\dd t.
\end{align*}
Using the local boundedness of $a$ and $b$, along with the boundedness of the partial derivative $\frac{\partial\varphi }{\partial x}$, and the boundedness of the sequence $\left\lbrace u_C^h \right\rbrace$ in $L_1((N,R)\times[0,T),\dd x \dd t)$, the left--hand side above can be bounded by a constant. It is easily seen that $\chi_{[N,\Xi^h(x)]}(w) X^h(w)$ converges pointwise to $\chi_{[N,x]}(w)w$ as $h \rightarrow 0$ and can be bounded by $R$ over our domain of interest. Therefore applying the Lebesgue dominated convergence theorem \cite[Theorem 1.8]{lieb01}, we get that
\eqref{equation510} converges to
\begin{align}\label{equation511}
&\int_{N}^R\int_{N}^{R}\chi_{[N,x]}(w)w\left(\int_0^T\int_{N}^{R} \chi_{[x,R]}(v)a(v)b(w|v)
u_C^R(v,t)\frac{\partial\varphi }{\partial x}(x,t)\dd v\dd t\right)\dd w \dd x\nonumber\\
&=\int_0^T\int_{N}^R\left(\int_{x}^{R}\int_{N}^{x}w a(v)b(w|v)u_C^R(v,t)\dd w\dd v \right)
\frac{\partial\varphi }{\partial x}(x,t)\dd x\dd t\nonumber\\
&=\int_0^T\int_{N}^R\mathcal{F}^R\left(xu_{C}^R\right)(x,t)\frac{\partial\varphi }{\partial x}(x,t)\dd x\dd t.
\end{align}
Therefore, in the limit as $h\rightarrow 0$, the third term of \eqref{equation506} coincides with the third term of \eqref{equation505}.

Now the fourth term from equation \eqref{equation506} is given fully by
\begin{align*}
&\sum_{n=0}^{M-1} \sum_{i=0}^{I_h-1}A_i\left(\sum_{j=1}^{N}j\tilde{B}_{j,i}\right)u_{C}^{n,i}\varphi_{i}^{n}\Delta x_{i} \Delta t\\
&=\sum_{n=0}^{M-1}\sum_{i=0}^{I_h-1}\int_{\tau_n}\int_{\Lambda_i}a^h(v)\left(\sum_{j=1}^{N}jb_j^h(v)\right)u_c^h(v,t)
\varphi(\xi^h(v),t)\dd v\dd t\\
&=\int_{0}^T\int_{N}^Ra^h(v)\left( \sum_{j=1}^{N}jb_j^h(v)\right)u_c^h(v,t)\varphi(\xi^h(v),t)\dd v\dd t.
\end{align*}
The pointwise convergence of $a^h$, $b_j^h$ and $\xi^h$ along with the continuity of $\varphi$ means that
\[a^h(v)\left(\sum_{j=1}^{N}jb_j^h(v)\right)\varphi(\xi^h(v),t)\rightarrow a(v)\left( \sum_{j=1}^{N}jb_j(v)\right)\varphi(v,t),\]
for all $t\in[0,T)$ and almost all $v\in(R,N)$ as $h\rightarrow0$. Since $a$ and $b_i$ are in $L_{\infty,loc}([N,\infty))$ and $\varphi$ is $C^2$ on
$[N,R]\times [0,T)$ with compact support (hence is a bounded function), the expressions on either side belong to $L_{\infty}((N,R)\times[0,T))$, with the left-hand side being uniformly bounded w.r.t. $h$. Hence, with $u_c^h\rightharpoonup u_C^R$ in $L_1((N,R)\times[0,T),\dd v\dd t)$, applying Theorem~\ref{theorem203}, yields
\begin{align}\label{equation512}
&\int_{0}^T \int_{N}^R a^h(v) \left( \sum_{j=1}^{N}jb_j^h(v)\right) u_c^h(v,t) \varphi(\xi^h(v),t) \dd v\dd t\nonumber\\
&\rightarrow\int_{0}^T\int_{N}^Ra(v)\left( \sum_{j=1}^{N}jb_j(v)\right)u_c^R(v,t)\varphi(v,t)\dd v\dd t\nonumber\\
&=\int_{0}^{T}\int_{N}^{R}S(vu_C^R)(v,t)  \varphi(v,t)\dd v\dd t.
\end{align}
Taken together \eqref{equation508}, \eqref{equation509}, \eqref{equation511} and \eqref{equation512} show that $u_C^R$ satisfies \eqref{equation505} for all $\varphi\in C_c^2([N,R]\times[0,T))$, and hence $u_C^R$ is a weak solution, as set out in Definition \ref{definition501}.
\end{proof}

\subsection{Discrete Fragmentation Regime: Convergence}

Now let us consider the discrete regime approximations. This is treated by a similar approach to the one we adopted for the continuous regime equation, but as a first step we establish a bound on the values $u_{D}^{n,i}$.

\begin{lemma}\label{lemma503}There exists a constant $C$, independent of $h$ and $R$, such that for all values of $n$ and $i$ we have
\[0\leq u_{D}^{n,i}\leq C.\]
\end{lemma}

\begin{proof}The nonnegativity of $u_{D}^{n,i}$ follows from Lemma \ref{lemma401}. We shall therefore concentrate on the upper bound. From Lemma~\ref{lemma402} we have, for all admissible $n$, that the following holds:
\begin{align*}
\sum_{i=0}^{I_{h-1}}x_i \Delta x_i u_{C}^{n,i} + \sum_{i=1}^N i u_{D}^{n,i}&=\sum_{i=0}^{I_{h-1}}x_i \Delta x_i u_{C}^{0,i} + \sum_{i=1}^N i u_{D}^{0,i}\\
&=\sum_{i=0}^{I_{h-1}}x_i \int_{\Lambda_i}c_0(x)\dd x + \sum_{i=1}^N i d_{0i}\\
&\leq\frac{3}{2}\sum_{i=0}^{I_{h-1}}\int_{\Lambda_i}c_0(x)\,x\dd x + \sum_{i=1}^N i d_{0i}\\
&\leq\frac{3}{2}\int_{N}^\infty c_0(x)\,x\dd x + \sum_{i=1}^N i d_{0i}=C<\infty.
\end{align*}
Therefore, for all $n$ and $i$ we have that
\[u_{D}^{n,i}\leq C,\]
where the constant $C$ is independent of the mesh parameter $h$ and the truncation parameter $R$.
\end{proof}

\begin{theorem}\label{theorem503}For each $i=1,\ldots,N$, the family $\left\lbrace u^h_{Di}\right \rbrace$ forms a sequentially weakly compact set in $L_{1}\left([0,T)\right)$, hence must have a weakly convergent subsequence.
\end{theorem}
\begin{proof}
The bound obtained in Lemma~\ref{lemma503} allows us to easily establish equiboundedness and equiintegrability in $L_{1}\left([0,T)\right)$ for each of the families $\left\lbrace u^h_{Di}\right \rbrace$  as follows:

\begin{equation}\label{equation513}
\left\|u_{Di}^h(\cdot)\right\|_{L_1\left([0,T)\right)}=\sum_{n=0}^{M-1}u^{n,i}_D\Delta t \leq \sum_{n=0}^{M-1}C\Delta t=CT.
\end{equation}

Now let $\Phi$ be any function of the nature described in Theorem~\ref{theorem204}. Since $\Phi$ is increasing, the established bound for $u^{n,i}_D$ allows us to deduce that
\begin{equation*}
\int_0^T\Phi(u_{Di}^h(t))\dd t=\sum_{n=0}^{M-1}\Phi(u^{n,i}_D)\Delta t \leq \sum_{n=0}^{M-1}\Phi(C)\Delta t=\Phi(C)T.
\end{equation*}

Hence each of the families $\left\lbrace u^h_{Di} \right \rbrace$ is equiintegrable. By the Dunford--Pettis theorem (Theorem~\ref{theorem205}), each of the families form a weakly sequentially compact set in $L_{1}\left([0,T)\right)$. As such, they all contain some weakly convergent subsequence.
\end{proof}

\begin{remark}
We note that what we seek is a collection of values $\left\lbrace h^j \right \rbrace$, such that all of  the sequences $\left\lbrace u^{h^j}_{Di} \right \rbrace$, for $i=1,\ldots,N$, converge weakly, as $j\rightarrow \infty$ and $h_j\rightarrow 0$. We achieve this by means of a diagonal argument, which we now outline. Knowing that the family $\left\lbrace u^h_{D1} \right \rbrace$ has a weakly convergent subsequence, let us denote the corresponding sequence of $h$-values by $\left\lbrace h^j \right \rbrace_{j=1}^\infty$ and consider the family $\left\lbrace u^{h^j}_{D2} \right \rbrace_{j=1}^\infty$. As this set satisfies the equiboundedness and equiintegrability conditions of the Dunford--Pettis theorem, it too must have a weakly convergent subsequence. Extracting this subsequence and denoting the corresponding $h$-values by $\left\lbrace h^{j_n} \right \rbrace_{n=1}^\infty$, we then have both
$\left\lbrace u^{h^{j_n}}_{D1} \right \rbrace$ and $\left\lbrace u^{h^{j_n}}_{D2} \right \rbrace$ converging (weakly) as $n\rightarrow \infty$ and $j_n\rightarrow \infty$ . We can continue this process, working through each of the families $\left\lbrace u^h_{Di} \right \rbrace$, until we have a set of common $h$-values, $\left\lbrace h_j^{\prime } \right \rbrace_{j=1}^\infty$, for which all the subsequences $\left\lbrace u^{h_j^{\prime }}_{Di} \right \rbrace_{j=1}^\infty$ are (weakly) convergent as $j\rightarrow \infty$ and $h_j^{\prime }\rightarrow 0$.

From now on these convergent subsequences are considered implicitly and we use $\left\lbrace u^h_{Di} \right \rbrace$ to denote said subsequences, unless otherwise stated. Let us denote the weak limit of $\left\lbrace u^h_{Di}\right \rbrace$ by $u_{Di}^R$ (note the upper case superscript notation for the limit).
\end{remark}

\subsection{Discrete Fragmentation Regime: Weak Solution}

Having established the convergence of our sequence of approximations we now aim to determine whether the limit produced provides a solution to the equation \eqref{equation403} and if so in what sense. As such, following on from Definition \ref{definition501}, we introduce
\begin{definition}\label{weakdisc}
We say that the function $u_{Di}^R$ is a weak solution of equation \eqref{equation403} if it satisfies
\begin{align}\label{equation515}
&\int_{0}^T u_{Di}^R(t)\frac{\ddn \phi}{\ddn t}(t)\dd t + d_{0_i}\phi(0) - \int_{0}^T a_iu_{Di}^R(t) \phi(t) \dd t\nonumber\\
&+ \int_{0}^T \sum_{j=i+1}^{N}a_jb_{i,j}u_{Dj}^R(t) \phi(t) \dd t + \int_{0}^T \int_{N}^{R} a(y)b_i(y)u_{C}^R(y,t)\phi(t)\dd y \dd t=0
\end{align}
for any $\phi \in C_c^2([0,T))$, where $C_c^2([0,T))$ is defined in an analogous fashion to  $C_c^2\left([N,R]\times[0,T)\right)$ from Definition \ref{definition501}.
\end{definition}

\begin{theorem}\label{theorem504}The functions $u_{Di}^R$ obtained as weak limits of the sequences $\left\lbrace u^h_{Di} \right \rbrace$ are indeed weak solutions of \eqref{equation403}, satisfying equation \eqref{equation515} for any $\phi \in C_c^2([0,T))$.
\end{theorem}

\begin{proof} For such a function $\phi$, let us denote its approximation over $\tau_n$ by $\phi^n$, which is defined as
\[\phi^n=\frac{1}{\Delta t}\int_{\tau_n}\phi(t)\dd t\hspace{2mm}\text{for}\hspace{1.5mm}n=0,\ldots,M-1,\]
and $\phi^M=0$. Multiplying \eqref{equation412} by $\phi^n$ and summing over $n$ from $0$ to $M-1$, gives us the following equality:
\begin{align*}
&\sum_{n=0}^{M-1}\left(u_{D}^{n+1,i}-u_{D}^{n,i}\right)\phi^n= -\sum_{n=0}^{M-1}a_i u_{D}^{n,i}\phi^n \Delta t \\
&+\sum_{n=0}^{M-1} \sum_{j=i+1}^{N} a_j b_{i,j} u_{D}^{n,j} \phi^n \Delta t+ \sum_{n=0}^{M-1} \sum_{j=0}^{I_h-1} A_j \tilde{B}_{i,j} u_{C}^{n,j}\phi^n
\Delta x_{j}\Delta t.
\end{align*}
Since $\phi$ is compactly supported, for sufficiently small $\Delta t$ we have $\phi^{M-1}=0$; then, further manipulation of the first term yields
\begin{align}\label{equation516}
&\sum_{n=0}^{M-1} u_{D}^{n+1,i} \left( \phi^{n+1} -\phi^n \right)+ u_{D}^{0,i}\phi^{0} - \sum_{n=0}^{M-1} a_i u_{D}^{n,i} \phi^n \Delta t\nonumber \\
&+\sum_{n=0}^{M-1}\sum_{j=i+1}^{N}a_jb_{i,j}u_{D}^{n,j}\phi^n\Delta t+\sum_{n=0}^{M-1}\sum_{j=0}^{I_h-1}A_j\tilde{B}_{i,j}u_{C}^{n,j}\phi^n
\Delta x_j\Delta t=0.
\end{align}
Looking more closely at the first term above we can rewrite it as
\begin{align*}
\sum_{n=0}^{M-1} u_{D}^{n+1,i} \left( \phi^{n+1} -\phi^n \right)
&=\underbrace{u_{D}^{M,i} \left( \phi^{M} -\phi^{M-1} \right)}_{=0}
+\sum_{n=0}^{M-2}  \int_{\tau_{n+1}} u_{Di}^h(t)   \frac{\phi(t)-\phi(t-\Delta t)}{\Delta t} \dd t\\
&=\int_{0}^T \chi_{[\Delta t,T)}(t) u_{Di}^h(t)   \frac{\phi(t)-\phi(t-\Delta t)}{\Delta t} \dd t.
\end{align*}
Assuming that $t\in(0,T)$ and $\Delta t\leq t$, then a Taylor series expansion of $\phi(t-\Delta t)$ about $t$ gives
\[\phi(t-\Delta t)=\phi(t)-\Delta t \frac{\ddn \phi}{\ddn t}(t)+O(\Delta t^2).\]
Therefore, we have
\begin{equation}\label{equation51515}
\chi_{[\Delta t,T)}(t)
\frac{\phi(t)-\phi(t-\Delta t)}{\Delta t}=\left\lbrace \begin{array}{cc}&\hspace{-5mm}\dfrac{\Delta t\frac{d \phi}{dl t}(x,t)+O(\Delta t^2)}{\Delta t}, \hspace{6mm}\Delta t \leq t\\ \vspace{-3mm}\\&\hspace{10.5mm}0,\hspace{21.2mm}\Delta t > t\end{array}\right..
\end{equation}
As such, the left-hand side of \eqref{equation51515} can be seen to converge pointwise to $\phi_t$ on $(0,T)$, as $h$, and by condition \eqref{equation404}, $\Delta t$ goes to $0$. By an analogous argument to that used for $\varphi$ and its derivatives, $\phi$ and its derivative $\phi_t$ must be bounded, therefore we can bound the left-hand side above, with the bound being uniform w.r.t $h$. Then, as ${u_D^h}_i\rightharpoonup {u_D^R}_i$ in $L_1\left([0,T)\right)$, applying Theorem~\ref{theorem203}, as before, gives us
\begin{equation}\label{equation517}
\int_{0}^T \chi_{[\Delta t,T)}(t) u_{Di}^h(t)\frac{\phi(t)-\phi(t-\Delta t)}{\Delta t} \dd t\rightarrow \int_{0}^T u_{Di}^R(t) \frac{\ddn \phi}{\ddn t}(t)\dd t.
\end{equation}
By definition, $u_{D i}^{n}={d_0}_i$, and since $\phi$ is $C^2$ with compact support, its derivative must be bounded, from which we deduce that
\[\phi^0=\frac{1}{\Delta t}\int_0^{\Delta t}\phi(t)\dd t\rightarrow\phi(0),\]
as $h$ goes to 0. Therefore
\begin{equation}\label{equation518}
u_{D}^{0 ,i}\phi^0\rightarrow{d_0}_i\phi(0)\hspace{1.5mm}\text{as}\hspace{1.5mm} h \rightarrow 0.
\end{equation}
By defining $b_{i,i}$ to be $-1$, we can combine the third and fourth terms of \eqref{equation516}, writing them as
\begin{align*}
\sum_{n=0}^{M-1} \sum_{j=i}^{N} a_j b_{i,j} u_{D}^{n,j} \phi^n \Delta t&=\sum_{n=0}^{M-1} \sum_{j=i}^{N} \int_{\tau_n}a_j b_{i,j} u_{Dj}^h(t) \phi(t)
\dd t\\
&=\int_0^T\sum_{j=i}^{N} a_j b_{i,j} u_{Dj}^h(t) \phi(t) \dd t,
\end{align*}
and since $u_{Dj}^h\rightharpoonup u_{Dj}^R$ in $L_1\left([0,T)\right)$, for each $j$, we have
\begin{align}\label{equation519}
&\int_0^T\sum_{j=i}^{N} a_j b_{i,j}u_{Dj}^h(t) \phi(t) \dd t \rightarrow \int_0^T\sum_{j=i}^{N} a_j b_{i,j} u_{Dj}^R(t) \phi(t) \dd t\nonumber\\
&=-\int_{0}^T a_iu_{Di}^R(t) \phi(t) \dd t+ \int_{0}^T \sum_{j=i+1}^{N}a_jb_{i,j}u_{Dj}^R(t) \phi(t) \dd t,
\end{align}
giving us the third and fourth terms of our weak formulation \eqref{equation515}. Rewriting the final term of our discrete relation, we get
\begin{align*}
\sum_{n=0}^{M-1} \sum_{j=0}^{I_h-1} A_j \tilde{B}_{i,j} u_{C}^{n,j}\phi^n\Delta x_j \Delta t&=\sum_{n=0}^{M-1} \sum_{j=0}^{I_h-1} \int_{\tau_n}
\int_{\Lambda_j} a^h(y)b_i^h(y)\phi(t) u_C^h(y,t)\dd y \dd t\\
&=\int_{0}^T\int_{N}^R a^h(y)b_i^h(y)\phi(t) u_C^h(y,t)\dd y \dd t.
\end{align*}
From Remark~\ref{convremark1} we have $a^h(y)$ and $b_i^h(y)$ converging pointwise to $a(y)$ and $b_i(y)$ respectively, and along with $\phi$ are bounded (uniformly with respect to $h$), a final application of Theorem~\ref{theorem203} allows us to deduce that
\begin{equation}\label{equation520}
\int_{0}^T\int_{N}^R a^h(y)b_i^h(y)\phi(t) u_C^h(y,t)\dd y \dd t\rightarrow\int_{0}^T \int_{N}^{R} a(y)b_i(y)u_{C}^R(y,t)\phi(t)\dd y \dd t,
\end{equation}
as the mesh size parameter $h\rightarrow0$. Taking the results \eqref{equation517}, \eqref{equation518}, \eqref{equation519} and \eqref{equation520}, we see that by letting $h\rightarrow0$ in \eqref{equation516} we obtain the weak formulation \eqref{equation515}, hence $u_{Di}^R$ is indeed a weak
solution of \eqref{equation403}.
\end{proof}

In this section we established the weak convergence of a subsequence of our sequence of approximate solutions as the mesh parameter was decreased to zero. The limits were shown to provide a set of weak solutions to the truncated equations \eqref{equation402} and \eqref{equation403}. However, there are a number of questions which remain unanswered which we seek to address in the following section.

\section{Uniqueness and Differentiability of Solutions }
In the previous section, we formed approximate solutions to a truncated version of our system. A subsequence of these approximations was shown to converge to a weak solution to our problem, as the underlying mesh was refined. This convergence of subsequences, rather than the full sequence, raises the possibility of nonunique solutions, with each convergent subsequence possibly offering a different solution. In this section we seek to address this, showing that any limits must coincide, providing a unique solution. Further, we would like to establish whether this solution may in fact display extra regularity, as we might expect from the results in \cite{baird18}. 

\subsection{Continuous Regime}

Returning to equation \eqref{equation40001}, we introduce the space $X_C^R=L_{1}\left((N,R),x\dd x\right)$ with the aim of recasting the equation as an abstract Cauchy problem, as was carried out for \eqref{equation301} in \cite{baird18}. Motivated by the terms appearing on the right-hand side, of \eqref{equation40001} we define the following linear operators $A_C^R$ and $B_C^R$ on the space $X_C^R$:
\begin{equation*}
(A_C^Rf)(x)=-a(x)f(x)\hspace{2mm} \text{and}\hspace{2mm} (B_C^Rf)(x)=\int_{x}^{R}a(y)b(x|y)f(y)\dd y\hspace{2mm} \text{for} \hspace{2mm}N<x<R.
\end{equation*}
with the respective domains
\[D(A_C^R)=\left\lbrace f\in X_C^R:A_C^Rf \in X_C^R\right\rbrace \hspace{6mm}D(B_C^R)=\left\lbrace f\in X_C^R:B_C^Rf \in X_C^R\right\rbrace,\]

Assuming that the functions $a$ and $b$ retain the properties imposed in Remark~\ref{conditions}, in particular $a\in L_{\infty,loc}\left([N,\infty)\right)$ and $b\in L_{\infty,loc}\left([N,\infty)\times [N,\infty)\right)$, then, the following property holds for the operators $A_C^R$ and $B_C^R$.

\begin{lemma} \label{lemma504}The operators $A_C^R$ and $B_C^R$ are bounded linear operators on the space $X_C^R$, with $\left\|B_C^Rf\right\|_{X_C^R}\leq\left\|A_C^Rf\right\|_{X_C^R}$ for all $f\in X_C^R$.
\end{lemma}
\begin{proof}We will first consider the operator $A_C^R$. Let $f\in X_C^R$; then we have
\begin{align*}
\left\|A_C^Rf\right\|_{X_C^R}&=\int_N^R \left|a(x) f(x)\right|\,x\dd x\\
&=\int_N^Ra(x)\left|f(x)\right|\,x\dd x.\\
&\leq \alpha(R)\int_N^R\left|f(x)\right|\,x\dd x= \alpha(R)\left\|f\right\|_{X_C^R}.
\end{align*}
Therefore $A_C^R$ is a bounded operator on the space $X_C^R$. The boundedness of $B_C^R$ in $X_C^R$ can be seen as follows. Let $f\in X_C^R$; then
\begin{align*}
\left\|B_C^Rf\right\|_{X_C^R}&=\int_{N}^{R}\left|\int_{x}^{R}a(y)b(x|y)f(y)\dd y\right|x\dd x\nonumber\\
&\leq\int_{N}^{R}\left(\int_{x}^{R}a(y)b(x|y)\left|f(y)\right|\dd y\right)x\dd x\nonumber\\                       &=\int_{N}^{R}a(y)\left|f(y)\right|\left(\int_{N}^{y}xb(x|y)\dd x\right)\dd y\\
                       &\leq \int_{N}^{R}a(y)\left|f(y)\right|y\dd y=\left\|A_C^Rf\right\|_{X_C^R}\leq \alpha(R)\left\|f\right\|_{X_C^R}.\nonumber
\end{align*}
The change in the order of integration can be justified by the nonnegativity of the integrand along with Tonelli's theorem. The inequality in going from the third to the fourth line comes as a result of the mass conservation condition~\eqref{equation303}.
\end{proof}
Equation \eqref{equation40001} is then recast as the following abstract Cauchy problem in the space $X_C^R$:
\begin{equation}\label{equation522}
\frac{\ddn}{\dd t}u_C^R(t)=\left(A_C^R+B_C^R\right)[u_C^R(t)],\hspace{3mm}t>0;\hspace{3mm}u_C^R(0)=c_{0}^R=\chi_{(N,R)}(x)c_{0}(x).
\end{equation}
Here $u_C^R$ denotes an $X_C^R$-valued function rather than the scalar-valued function of two variables from the previous section. However due to the relationship between the spaces $L_1\left(I,L_1(\Omega,\dd \mu)\right)$ and $L_1\left(\Omega\times I,\dd \mu \dd t \right)$, we may switch between the two, with each $L_1$-valued solution to \eqref{equation522} providing us with a scalar-valued solution to \eqref{equation40001} and vice versa.

\begin{lemma}\label{lemma505} The operator $\left(A_C^R+B_C^R\right)$ generates a uniformly continuous semigroup of positive contractions on $X_C^R$.
\end{lemma}
\begin{proof}As a bounded linear operator on the space $X_C^R$, the sum $A_C^R+B_C^R$ generates a uniformly continuous semigroup $\left(T_R(t)\right)_{t\geq0}$ on $X_C^R$, \cite[Chapter 1, Theorem 3.7 and Chapter 2, Corollary 1.5]{engel00}. Furthermore, the Kato--Voigt perturbation theorem, \cite[Corollary 5.17]{banasiak06} as applied in \cite[Theorem 3.2]{baird18}, is readily utilised in the case $\left(A_C^R+B_C^R,X_C^R\right)$ to give us an `extension' of $\left(A_C^R+B_C^R,X_C^R\right)$ as a generator of a substochastic semigroup. Now as $A_C^R+B_C^R$ is defined and bounded on all of $X_C^R$, this extension must be $A_C^R+B_C^R$ itself and by \cite[Theorem 2.6]{pazy83}, the substochastic semigroup generated must be $\left(T_R(t)\right)_{t\geq0}$. Hence $\left(T_R(t)\right)_{t\geq0}$ is a uniformly continuous semigroup of positive contractions.
\end{proof}

From standard results concerning strongly continuous semigroups, \cite[Theorem 2.40 and Theorem 2.41]{morante98}, the existence of the semigroup $\left(T_R(t)\right)_{t\geq0}$ on $X_C^R$, provides a unique strong solution to equation \eqref{equation522}, given by $u_C^R(t)= T_R(t)c_{0}^R$. Additionally, by \cite[Chapter 2, Proposition 6.4]{engel00} this is also a unique mild solution, satisfying an equation of the form
\begin{equation}\label{equation523}
u_C^R(t)=c_{0}^R+\left(A_C^R+B_C^R\right)\int_0^tu_C^R(s)\dd s=c_{0}^R+\int_0^t\left(A_C^R+B_C^R\right)u_C^R(s)\dd s, \hspace{3mm}(t\geq 0).
\end{equation}
We are able to take the operator $A_C^R+B_C^R$ inside the integral as a consequence of its boundedness by applying \cite[Proposition 1.1.7.]{arendt01}. Hence our equation \eqref{equation523} corresponds with the mild solution form of \cite[Definition 1.10]{canizo06}. The article \cite{canizo06} also provides a notion of an $X_C^R$-valued weak solution of equations of type \eqref{equation522}, which we outline here for our specific example.

\begin{definition}\label{definition282} The function  $u_C^R:[0,T)\longrightarrow X_C^R$ is a \emph{weak} solution of \eqref{equation522}, if for all $\phi$ in $L_{\infty}\left((N,R)\right)$ (the dual space of $X_C^R$), we have that $t\rightarrow \langle u_C^R(t),\phi\rangle$ is locally integrable in $(0,T)$ and
\begin{equation}\label{equation282}
\int_0^T\langle u_C^R(s),\phi \rangle \frac{\ddn}{\ddn s}\psi(s)\dd s =-\langle f_0,\phi \rangle\psi(0) - \int_0^T \left\langle \left(A_C^R+B_C^R\right)u_C^R(s),\phi \right\rangle \psi(s)\dd s,
\end{equation}
for all $\psi \in C^\infty\left([0,T)\right)$ with compact support, where $\langle g,\phi\rangle$ denotes the duality pairing of $g$ and $\phi$.
\end{definition}
\begin{remark}\label{remark284}If $D\subseteq {X_C^R}'$ is dense in the weak-$\ast$ topology, then it is sufficient to show that \eqref{equation282} holds for all $\phi \in D$ to establish $u_C^R:[0,T)\rightarrow X_C^R$ as a \emph{weak} solution of \eqref{equation522}; see \cite[Definition 1.9]{canizo06}.
\end{remark}

The results so far have provided us with the existence of unique strong and mild solutions to \eqref{equation522}. We now show that, for our case, any mild solution satisfying \eqref{equation523} must necessarily be a weak solution as in Definition~\ref{definition282} and vice-versa, providing the existence and uniqueness of a weak solution.
\begin{theorem}\label{theorem505}The function  $u_C^R:[0,\infty)\longrightarrow X_C^R$ provided by the semigroup $\left(T_R(t)\right)_{t\geq0}$ is the unique weak solution to equation \eqref{equation522}, satisfying Definition~\ref{definition282} over any time interval $[0,T)$ where $T<\infty$.
\end{theorem}
\begin{proof}From the analysis above we have the semigroup $\left(T_R(t)\right)_{t\geq0}$ providing a unique mild (strong) solution $u_C^R:[0,\infty)\longrightarrow X_C^R$ to equation \eqref{equation522}. The result \cite[Theorem 1.20]{canizo06} tells us that if the right-hand side of \eqref{equation522} is integrable then, given either a weak or mild solution to our evolution equation \eqref{equation522}, the solution can be modified on a set of measure zero to obtain a solution of the other form. Therefore in our case, if the conditions of \cite[Theorem 1.20]{canizo06} hold, the existence of a unique mild solution will provide a unique weak solution (as given in  Definition~\ref{definition282}) on each finite time interval $[0,T)$, with this solution being given by the semigroup $\left(T_R(t)\right)_{t\geq0}$.

Let $0<T<\infty$; and let us consider the following integral:
\begin{align*}
&\int_0^T\left\|(A_C^R+B_C^R)u_C^R(s)\right\|_{X_C^R}\dd s\\
&=\int_0^T\hspace{-.5mm} \int_N^R \bigg|-a(x)(u_C^R(s))(x)
+\int_{x}^{R}a(y)b(x|y)(u^R_C(s))(y)\dd y\bigg|\,x\dd x \dd s\\
&\leq \alpha(R)\hspace{-.8mm}\int_0^T \hspace{-1.5mm}\int_N^R \hspace{-.5mm}\left|(u_C^R(s))(x)\right|\,x\dd x\dd s
+\alpha(R)\hspace{-.7mm}\int_0^T\hspace{-1.5mm}\int_N^R\hspace{-1.5mm}\int_{x}^{R}b(x|y)\left|(u_C^R(s))(y)\right|\dd y\,x\dd x\dd s\\
&= \hspace{-.5mm}\alpha(R)\hspace{-1mm}\int_0^T\hspace{-1.5mm}\int_N^R  \hspace{-1mm}\left|(u_C^R(s))(x)\right|\,x\dd x\dd s
+\alpha(R) \hspace{-1mm}\int_0^T  \hspace{-1.5mm}\int_N^R \left|(u_C^R(s))(y)\right|\int_{N}^{y}\hspace{-.7mm}b(x|y)\,x\dd x\dd y\dd s\\
&\leq 2\alpha(R)\int_0^T \int_N^R \left|(u_C^R(s))(x)\right|\,x\dd x\dd s= 2\alpha(R)\int_0^T \left\|T_R(s)c_0^R\right\|_{X_C^R}\dd s\leq 2\alpha(R)T \left\|c_0^R\right\|_{X_C^R}<\infty.
\end{align*}
In going to the final line, the bounding of the inner integral of the second term from the previous line comes from the mass conservation condition~\eqref{equation303},
whilst the subsequent norm inequality relies on the fact that the semigroup $\left(T_R(t)\right)_{t\geq0}$ consists of contractions.

By \cite[Chapter 2, Theorem 2 and Theorem 4]{diestel77}, the above bound implies that the right-hand side of equation \eqref{equation522} is integrable over the interval $[0,T)$ and so by \cite[Theorem 1.20]{canizo06}, any mild or weak solutions must agree (up to sets of measure zero). Therefore equation \eqref{equation522} has a unique weak solution given by $u_C^R(t)=T_R(t)c_0^R$,
which is in fact a strong solution.
\end{proof}

Let $u_C^R(x,t)$ be a scalar representation of the semigroup solution $(u_C^R(t))(x)=\left(T_R(t)c_0^R\right)(x)$. Then
applying Definition \ref{definition282} to our example and noting the equivalence of \eqref{equation40001} and \eqref{equation402} as detailed fully in \cite[Appendix C]{bairdphd}, we get that equation \eqref{equation282} is equivalent to
\begin{align}\label{equation524}
&\int_{0}^{T}\int_{N}^{R} xu_C^R(x,t)\frac{\partial\varphi}{\partial t}(x,t)\dd x\dd t
+\int_{N}^{R}xu_C^R(x,0)(x)\varphi(x,0)\dd x\nonumber\\
&=\int_{0}^{T}\int_{N}^{R}\mathcal{F}^R\left(xu_C^R(x,t)\right)\hspace{-.6mm}(x,t)\hspace{1mm}\frac{\partial\varphi}{\partial x}(x,t) \dd x \dd t\nonumber\\
&\hspace{45mm}+\int_{0}^{T}\int_{N}^{R}S(xu_C^R(x,t))(x,t)\varphi(x,t)\dd x\dd t,
\end{align}
for all $\varphi$ of the form $\varphi(x,t)=\phi(x)\psi(t)$ where $\phi\in C^{\infty}_c\left((N,R)\right)$ and $\psi \in C^{\infty}_c\left([0,T)\right)$, due to the weak-$\ast$ density of $C^{\infty}_c\left((N,R)\right)$ in $L_{\infty}\left((N,R)\right)$, via Remark \ref{remark284}.

Having determined a one-to-one correspondence between scalar-valued weak solutions satisfying \eqref{equation524} and $X_C^R$-valued strong solutions of the abstract Cauchy problem \eqref{equation522}, we are now in a position to establish the uniqueness and differentiability of the weak solutions of the previous section.

\begin{theorem}\label{theorem506}The weak solution obtained as the limit of the sequence of approximate solutions for the continuous
regime in Theorems~\ref{theorem501} and ~\ref{theorem502} is unique, continuously differentiable with respect to $t$ on any interval $[0,T)$
and satisfies equation \eqref{equation40001}
directly, except perhaps on a set of measure zero.
\end{theorem}
\begin{proof} It is easily seen that any scalar-valued function $u_C^R(x,t)$ satisfying Definition~\ref{definition501} will immediately
satisfy the equation \eqref{equation524} above. From any such scalar-valued function we may define a function
$u_C^R:[0,T)\longrightarrow X_C^R$ via $\left(u_C^R (t) \right)(x):=u_C^R(x,t)$, for almost all $(x,t)\in (N,R)\times[0,T)$. Since the
scalar function satisfies equation \eqref{equation524}, the $X_C^R$-valued function must provide a weak solution, as defined in
Definition~\ref{definition282}, to the abstract Cauchy problem \eqref{equation522}. By Theorem \ref{theorem505}, this weak solution must
necessarily also be the unique (up to sets of measure zero) strong solution of \eqref{equation522}. Hence for each scalar weak solution
$u_C^R(x,t)$ satisfying Definition~\ref{definition501} there is a unique corresponding $X_C^R$-valued strong solution $u_C^R$ to the
associated abstract Cauchy problem, whereby the original function $u_C^R(x,t)$ is one scalar representation of the strong solution.
Since the strong solution is unique, as are its scalar representations (up to measure zero), it follows that the weak solution must be
unique, up to a set of measure zero. Further, as a representation of a strongly differentiable $X_C^R$-valued function, by \cite[Theorem 2.40]{banasiak06}, $u_C^R(x,t)$ is continuously differentiable with respect to $t$, except perhaps on a set of zero measure. Further, by a similar argument as applied at the end of \cite[Theorem 8.3]{banasiak06} or directly as in \cite[Theorem 3.2.7.]{bairdphd}, $u_C^R(x,t)$ can be seen to directly satisfy equation \eqref{equation40001} almost everywhere.
\end{proof}

This result greatly strengthens those of the previous section, where before we had only the existence of a (weakly) convergent
subsequence and the possibility of the numerical scheme converging to multiple weak solutions. We now know that the limit solution must necessarily be unique, continuously differentiable with respect to $t$ and a solution, in the classical sense, of the truncated
fragmentation equation.

\subsection{Discrete Regime}

Recalling the truncated discrete regime equation \eqref{equation403}, we have for $i=1,\ldots,N$:
\begin{align}\label{equation550}
\hspace{-80mm}\frac{\ddn u_{Di}^R(t)}{\ddn t}
&\hspace{-.7mm}=\hspace{-.7mm}-a_iu_{Di}^R(t)+\hspace{-1.9mm}\sum_{j=i+1}^{N}\hspace{-1.2mm}a_jb_{i,j}u_{\small{D}}^R(t)_j+\hspace{-1mm}\int_{N}^{R}a(y)b_i(y)u_{\small{C}}^R(y,t)\dd y,\hspace{.7mm}\hspace{1.1mm}t>0,\\
 u_{D}^R(0)&=d_{0}.\nonumber
\end{align}
With the aim of recasting these equations as an inhomogeneous abstract Cauchy problem as in \cite{baird18}, we introduce the space $X_D=\mathbb{R}^N$, equipped with the weighted norm:
\[\left\|v\right\|_{X_D}=\sum_{j=1}^{N}j|v_j|,\hspace{2mm} \text{where} \hspace{2mm}v=(v_1,\ldots,v_N).\]
The equations \eqref{equation550} then become
\begin{equation}\label{equation551}
\frac{\ddn}{\ddn t}u_D^R(t)=(A_D+B_D)[u_D^R(t)]+C_R[u_C^R(t)],\hspace{3mm}t>0;\hspace{3mm}u_D^R(0)=d_{0},
\end{equation}
where $A_D$ and $B_D$ are defined on $X_D$, by
\[(A_Dv)_i=-a_iv_{i}\hspace{2mm}\text{and}\hspace{2mm}(B_Dv)_i=\sum_{j=i+1}^{N}a_jb_{i,j}v_{j},\hspace{2mm}\text{for}\hspace{2mm}i=1,
\ldots,N,\]
$u_C^R$ is the truncated semigroup solution of \eqref{equation522} and where $C_R \hspace{-0.5mm}:\hspace{-0.5mm} D(C_R)\subseteq X_C^R \rightarrow X_D $ is given by
\[(C_Rf)_i=\int_{N}^{R}a(y)b_i(y)f(y)\dd y, \hspace{3mm} D(C_R)=\left\lbrace f\in X_C^R:C_Rf \in X_D\right\rbrace=X_C^R,\]
for $i=1,\ldots,N$. The fact that $D(C_R)=X_C^R$ is a consequence of the $L_{\infty,loc}$ boundedness of $a$ and each $b_i$. Recall that since the space $X_D$ is finite-dimensional, the operators $A_D$ and $B_D$ are bounded and therefore by \cite[Chapter 1, Theorem 3.7 and Chapter 2, Corollary 1.5]{engel00} their sum must generate a uniformly continuous semigroup $\left(T(t)\right)_{t\geq0}$ on $X_D$. We then consider the term
$C_R[u_C^R(t)]$ as a perturbation and \eqref{equation551} as an inhomogeneous abstract Cauchy problem. The following lemma establishes the differentiability of this perturbation term, a property that we will require in showing the existence of solutions of equation \eqref{equation551}, along with determining their nature.

\begin{lemma}\label{lemma601} The term $C_R[u_C^R(t)]$ from \eqref{equation551} is strongly differentiable (in the space $X_D$) with respect to $t$, at almost all points of $[0,T]$. Furthermore its derivative is given by $C_R[\frac{\ddn}{\ddn t}u_C^R(t)]$.
\end{lemma}
\begin{proof} Recalling the $L_{\infty,loc}$ boundedness of $a$ and the bound $b_i(y)\leq y$ for each $i\in \left\lbrace1,2,\ldots,N\right\rbrace$, which is easily derived from \eqref{equation303}, we have that

\begin{align*}
&\left|\frac{\left(C_R[u_C^R(t+h)]\right)_i-\left(C_R[u_C^R(t)]\right)_i}{h}-\left(C_R\left[\frac{\ddn}{\ddn t}u_C^R(t)\right]\right)_i\right|\\
&=\left| \int_{N}^{R} a(y)b_i(y)\left(\frac{\left(u_C^R(t+h)\right)\hspace{-1.5mm}(y)-\left(u_C^R(t)\right)\hspace{-1.5mm}(y)}{h}
-\left[ \frac{\ddn}{\ddn t}u_C^R(t)\right]\hspace{-1mm}(y)\right) \dd y\right|\\
&\leq \alpha(R)\int_{N}^{R}\left|\frac{\left(u_C^R(t+h)\right)(y)-\left(u_C^R(t)\right)(y)}{h}-\left[ \frac{\ddn}{\ddn t}
u_C^R(t)\right]\hspace{-1mm}(y)\right|\,y\dd y\\
&=\alpha(R)\left\|\frac{u_C^R(t+h)-u_C^R(t)}{h}-\frac{\ddn}{\ddn t}u_C^R(t)\right\|_{X_C^R},
\end{align*}
where $\alpha(R)$ is the essential supremum of $a$ over $[N,R]$. From the differentiability of $u_C^R$, by
letting $h \rightarrow 0$ on both sides of the above calculation, we may deduce that $C_R[u_C^R(t)]$ is differentiable (in the space $X_D$) at almost
all points of $[0,T]$, with derivative $C_R[\frac{\ddn}{\ddn t}u_C^R(t)]$.
\end{proof}

Having shown the differentiability of the perturbation term $C_R[u_C^R(t)]$, we now look at its derivative more closely, showing that
it is integrable, belonging to the space $L_1((0,T),X_D)$ and in doing so establish the existence of a unique strong solution to
equation~\eqref{equation551}.

\begin{theorem}\label{theorem602}The derivative of $C_R[u_C^R(t)]$ belongs to the space $L_1((0,T),X_D)$.
As such equation~\eqref{equation551} has a unique strong solution, which is given by
\begin{equation}\label{solutionform}
u_{D}^R(t)=T(t)d_{0}+\int_0^t T(t-s)C_R[u_C^R(s)]\dd s.
\end{equation}
\end{theorem}
\begin{proof}If we take the $X_D$-norm of the derivative $\frac{\ddn}{\ddn t}C_R[u_C^R(t)]$ established in the previous lemma and integrate from $0$ to $T$, then we obtain
\begin{align}\label{equation604}
&\int_0^T \left\|\frac{\ddn}{\ddn t}C_R\left[u_C^R(t)\right]\right\|_{X_D}\,\hspace{-2mm}\dd t
=\int_0^T \left\|C_R\left[\frac{\ddn}{\ddn t}u_C^R(t)\right]\right\|_{X_D}\hspace{-1mm}\dd t\nonumber\\
&=\int_0^T \sum_{i=1}^Ni\left|\int_N^R a(y)b_i(y)\left[\frac{\ddn}{\ddn t}u_C^R(t)\right](y)\dd y\right|
\dd t\nonumber\\
&\leq\alpha(R)N^2\int_0^T\left\lbrace \int_N^R \left|\left[\frac{\ddn}{\ddn t}u_C^R(t)\right](y)\right|\,y\dd y
\right\rbrace\dd t\nonumber\\
&=\alpha(R)N^2 \int_0^T \left\| \frac{\ddn} {\ddn t}u_C^R(t) \right\|_{X_C^R}\hspace{-2mm}\dd t.
\end{align}
Recalling the abstract Cauchy problem \eqref{equation522}, since the operators $A_C^R$ and $B_C^R$ are bounded and since $u_C^R$ is given by a contraction semigroup, we have
\begin{align*}
\left\|\frac{\ddn}{\ddn t}u_C^R(t)\right\|_{X_C^R}\leq\left\|A_C^R+B_C^R\right\|\left\|u_C^R(t)\right\|_{X_C^R}\leq M \left\|c_0\right\|_{X_C^R},
\end{align*}
where $M$ is a constant such that $\left\|A_C^R+B_C^R\right\|\leq M$. Inserting this into \eqref{equation604} gives us
\begin{align*}
\int_0^T\left\|\frac{\ddn}{\ddn t}C_R[u_C^R(t)]\right\|_{X_D}\,\hspace{-4mm}\dd t&\leq \alpha(R)N^2\int_0^T M \left\|c_0\right\|_{X_C^R}\dd s\leq \alpha(R)N^2TM \left\|c_0\right\|_{X_C^R}<\infty.
\end{align*}
Therefore the derivative of $C_R[u_C^R(t)]$ belongs to the space $L_1((0,T),X_D)$; hence, by \cite[Chapter 4, Corollary 2.2 and Corollary 2.10]{pazy83}, the equation~\eqref{equation551} has a unique strong solution $u_{D}^R:[0,T)\rightarrow X_D$ given by \eqref{solutionform}.
\end{proof}

Having established the existence of a unique strong solution to equation~\eqref{equation551} given by \eqref{solutionform}, this solution must also provide us with a unique mild solution to our equation. Now we consider the possibility of weak solutions, as defined in Definition~\ref{definition282}. We aim to show that any weak solution of equation~\eqref{equation551} must also be a mild solution (permitting changes on sets of measure zero), and hence the weak solution must be unique and differentiable.

\begin{theorem}\label{theorem603}Given an integrable weak solution $u_{D}^R:[0,T)\rightarrow X_D$ of equation~\eqref{equation551} as defined in
Definition~\ref{definition282}, then it must also be a strong solution. Therefore any integrable weak solution must be unique up to sets of measure zero and differentiable in $X_D$.
\end{theorem}

\begin{proof}Considering the right-hand side of equation \eqref{equation551}, taking the norm in $X_D$ and integrating from
$0$ to $T$ gives us
\begin{align*}
&\int_0^T \left\|(A_D+B_D)[u_D^R(t)]+C_R[u_C^R(t)]\right\|_{X_D} \dd t\\
&\leq \int_0^T \left\|(A_D+B_D)[u_D^R(t)]\right\|_{X_D} \dd t+\int_0^T \left\|C_R[u_C^R(t)]\right\|_{X_D} \dd t\\
&= \int_0^T \left\|(A_D+B_D)[u_D^R(t)]\right\|_{X_D} \dd t+\int_0^T \sum_{i=0}^Ni\left|\int_N^R a(y)b_i(y)(u_C^R(t))(y)\dd y
\right|\dd t\\
&\leq   \int_0^T\left\|A_D+B_D\right\|\left\|u_D^R(t)\right\|_{X_D}\dd t
+\alpha(R)N^2\int_0^T \int_N^R \left|(u_C^R(t))(y)\right|\,y\dd y
\dd t\\
&\leq \left\|A_D+B_D\right\|  \int_0^T\left\|u_D^R(t)\right\|_{X_D}\dd t+\alpha(R)N^2\int_0^T \left\|u_C^R(t)\right\|_{X_C^R} \dd t.
\end{align*}
The assumption that $u_D^R(t)$ is integrable and  \cite[Chapter 2, Theorem 2 and Theorem 4]{diestel77} allow us to deduce that the first of these integrals must be finite,
whilst recalling that $u_C^R(t)$ was given by a contraction semigroup immediately enables us to bound the second integral above.
Hence,  \cite[Chapter 2, Theorem 2 and Theorem 4]{diestel77}, the right-hand side of equation~\eqref{equation551} is integrable and therefore \cite[Theorem 1.20]{canizo06}
tells us that the weak solution $u_D^R(t)$ (allowing for changes on sets on measure zero) must also be a mild solution.
Since, by Theorem~\ref{theorem602}, equation~\eqref{equation551} has a unique mild solution which is in fact a strong solution, the weak solution
$u_D^R(t)$ we started with must agree with the strong solution (up to sets of measure zero) and therefore is unique and differentiable.
\end{proof}

Having established that any integrable weak solution of equation~\eqref{equation551}, in the sense of Definition~\ref{definition282},
is also a strong solution, we now set out to prove that the solutions of equation~\eqref{equation550} obtained previously as the limit of our numerical scheme, provide us with such a weak solution and in the process establish their uniqueness and differentiability.

\begin{lemma}\label{lemma602}The weak solutions $u_{Di}^R(t)$ to the equations~\eqref{equation550}, obtained from our numerical scheme, when taken as the components of $u_{D}^R:[0,T)\rightarrow X_D$, produce a $u_{D}^R$ which is integrable.
\end{lemma}
\begin{proof}The weak solutions $u_{Di}^R$ to the equations \eqref{equation550}, constructed in the previous section, were obtained as the weak limits in $L_1(0,T)$ of the sequences $\left\lbrace u^h_{Di} \right \rbrace$ as we let $h\searrow 0$. By the weak lower semicontinuity of the norm \cite[Theorem 2.11]{lieb01}, and using the bound \eqref{equation513}, we obtain
\[\|u_{Di}^R\|_{L_1(0,T)}\leq\liminf_{h\rightarrow 0}\left\|u^h_{Di} \right\|_{L_1(0,T)}\leq CT,\]
where $C$ denotes the constant from \eqref{equation513}. If we take the components of the function $u_{D}^R:[0,T)\rightarrow X_D$
to be given by $u_{Di}^R$ for $i=1,\ldots,N$ then we get
\begin{align*}
\int_0^T \|u_{D}^R(t)\|_{X_D}\dd t=\int_0^T \sum_{i=1}^Ni\left|u_{Di}^R(t)\right|\dd t
\leq N\sum_{i=1}^N \underbrace{\int_0^T \left|u_{Di}^R(t)\right|\dd t}_{=\|u_{Di}^R\|_{L_1(0,T)}}
\leq N^2CT<\infty.
\end{align*}
Therefore the function $u_{D}^R:[0,T)\rightarrow X_D$ formed by taking $u_{Di}^R$ as its $i^{th}$ component is integrable.
\end{proof}

\begin{theorem}\label{theorem604}The weak solutions $u_{Di}^R:[0,T)\rightarrow \mathbb{R}$ for $i=1,\ldots,N$ of the equations~\eqref{equation515} obtained in the previous section agree with the components of the strong solution established in Theorem~\ref{theorem602} and hence are unique (up to sets of measure zero) and differentiable.
\end{theorem}

\begin{proof}Let us consider our abstract equation~\eqref{equation551} with the aim of rewriting it in a weak formulation as in \eqref{equation282}. First let us note that the dual space of $X_D$ is $\mathbb{R}^N$ and the duality pairing $\langle \cdot,\cdot\rangle$ appearing in \eqref{equation282} is given by the standard inner product on $\mathbb{R}^N$. Let $u_{D}^R:[0,T)\rightarrow X_D$ be a weak solution of equation~\eqref{equation551} with the components $u_{Di}^R:[0,T)\rightarrow \mathbb{R}$ for $i=1,\ldots,N$. Then, in this case, the function $F:(0,T)\rightarrow X_D$ appearing in \eqref{equation282} is given componentwise by
\[F_i(t)=-a_iu_{Di}^R(t)+\hspace{-1.9mm}\sum_{j=i+1}^{N}\hspace{-1.2mm}a_jb_{i,j}u_{Dj}^R(t)
+\hspace{-1mm}\int_{N}^{R}a(y)b_i(y)(u_C^R(t))(y)\dd y,\]
for $i=1,\ldots,N$. Therefore the weak formulation of equation~\eqref{equation551} becomes:
\begin{align*}
&\sum_{i=1}^N \phi_i \int_{0}^T \hspace{-1.5mm}u_{Di}^R(t) \frac{\ddn}{\ddn t} \psi(t) \dd t
=-\sum_{i=1}^N \phi_i{d_{0}}_i\psi(0)
+\sum_{i=1}^N \phi_i\int_{0}^T\hspace{-1.5mm} a_iu_{Di}^R(t) \psi(t) \dd t\\
&-\sum_{i=1}^N \phi_i\int_{0}^T\hspace{-1.5mm} \sum_{j=i+1}^{N}\hspace{-1.5mm}a_jb_{i,j}u_{Dj}^R(t) \psi(t) \dd t
-\sum_{i=1}^N \phi_i\int_{0}^T \hspace{-1mm}\int_{N}^{R}a(y)b_i(y)u_{C}^R(y,t)\psi(t)\dd y \dd t,
\end{align*}
for any $\phi=\left(\phi_1,\ldots,\phi_N\right)\in \mathbb{R}^N$ and $\psi \in C_c^\infty\left([0,T)\right)$, where $u_{C}^R(\cdot,t)$ is the (unique) real-valued representation of the $X_C^R$-valued $u_{C}^R:[0,T)\rightarrow X_C^R$.

Comparing this with equation~\eqref{equation515} of the previous section, it is easily seen that the $u_{Di}^R$ obtained there provide
us with a solution to the above equation and so taking these $u_{Di}^R:[0,T)\rightarrow \mathbb{R},$ $i=1,\ldots,N$, as the components of
an $X_D$-valued function $u_{D}^R:[0,T)\rightarrow X_D$ provides us with an integrable weak solution to equation~\eqref{equation551}. As any such
$u_{D}^R:[0,T)\rightarrow X_D$ must be unique and differentiable in $X_D$, the components $u_{Di}^R:[0,T)\rightarrow \mathbb{R}$ must be unique
(up to sets of measure zero) and differentiable in the traditional sense.
\end{proof}
\noindent In this article we have shown the convergence of our approximate solutions to a weak solution of the truncated problem given by equations \eqref{equation40001} and \eqref{equation403}, and the equivalence of this weak solution to the unique strong/classical solution of the truncated problem. However, it is possible using standard arguments along the lines of \cite[Section 8.3.2]{banasiak06} to show these truncated solutions converge, in the sense of the appropriate space, to the unique strong/classical solution of the untruncated problem as given by \eqref{equation301} and \eqref{equation302}, whose existence was established in \cite{baird18}. For further details of this convergence in the specific case of the mixed discrete-continuous model, the reader is directed to consult \cite[Chapter 6]{bairdphd}. Furthermore, for an experimental study of this convergence, the factors influencing it and therefore the selection of a suitable value for the truncation parameter $R$, the reader is directed to \cite[Section 7.5]{bairdphd}.

\section{Numerical Experiments}\label{sectionsexperiments}
To assess the efficacy of our numerical scheme, we tested it on the power law model as set out in \cite[Section 7]{baird18}, where the continuous equation was defined by
\begin{equation*}
a(x)=x^\alpha,\hspace{4mm}\alpha\in\mathbb{R},\hspace{5mm}\text{and}\hspace{5mm}b(x|y)=(\nu+2)\frac{x^\nu}{y^{\nu+1}},\hspace{4mm}-2<\nu\leq0.
\end{equation*}
The discrete equation was specified by the following values for $a_i$ and $b_{i,j}$
\[
a_i=\left\{ \begin{aligned}
\hspace{2mm} 0\hspace{4mm} &\textup{for} \hspace{2mm} i=1, \\
\hspace{2mm} i^\alpha\hspace{4mm} &\textup{for} \hspace{2mm} i=2,\ldots,N,
\end{aligned} \right.\hspace{5mm}
b_{i,j}=\frac{2}{j-1}, \hspace{5mm}i=1,\ldots,N-1,\hspace{1.2mm}j=i,\ldots,N,
\]
and the continuous to discrete distribution functions $b_i(y)$ were given by
\begin{align*}
b_i(y)=\frac{i^{\nu+2}-(i-1)^{\nu+2}}{iy^{\nu+1}}, \hspace{5mm}y>N,\hspace{1.2mm}i=1,\ldots,N.
\end{align*}

\begin{figure}[h]
\begin{minipage}[b]{0.45\linewidth}
\centering
\includegraphics[scale=0.5]{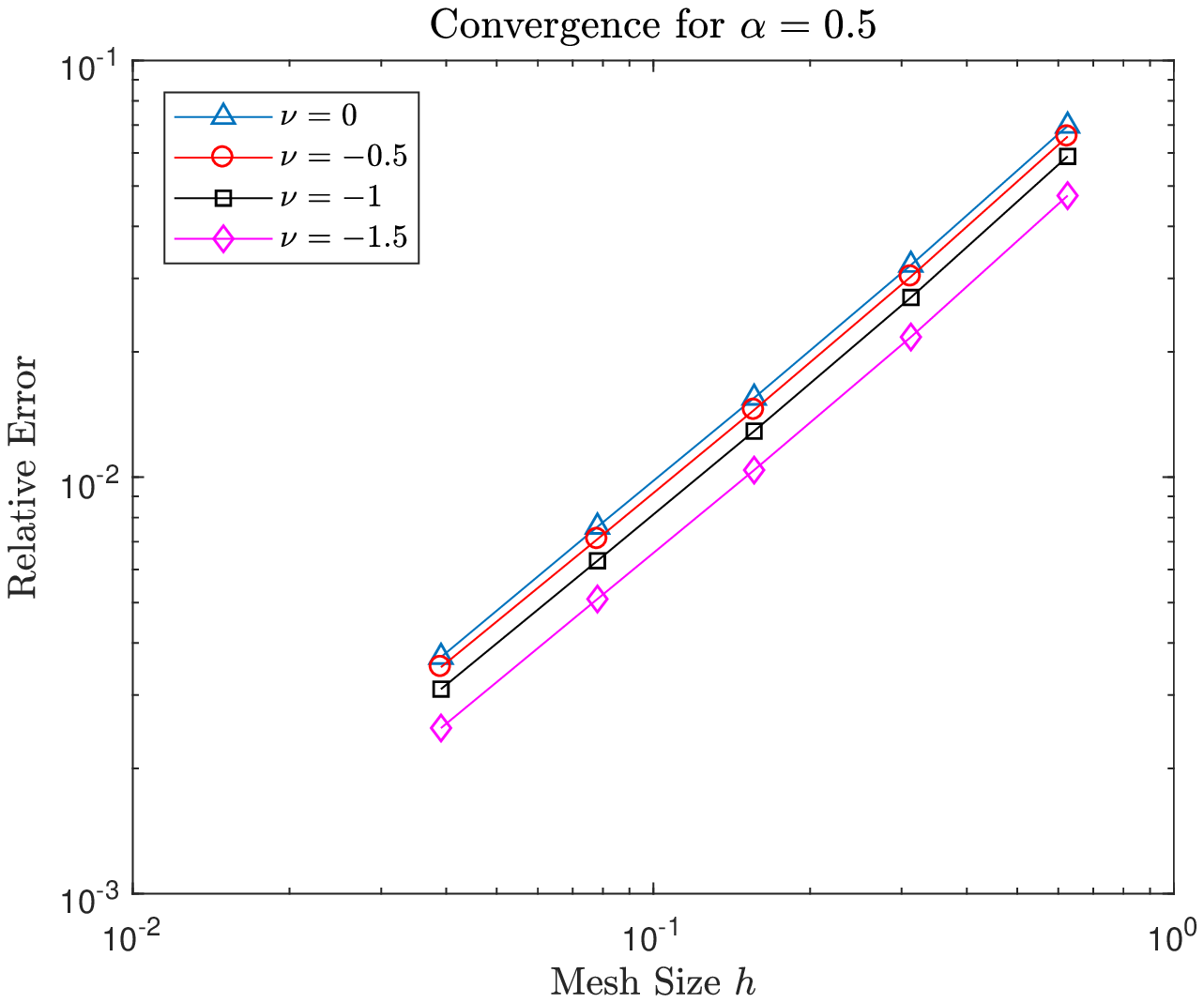}\\\hspace{0mm}\\
\includegraphics[scale=0.5]{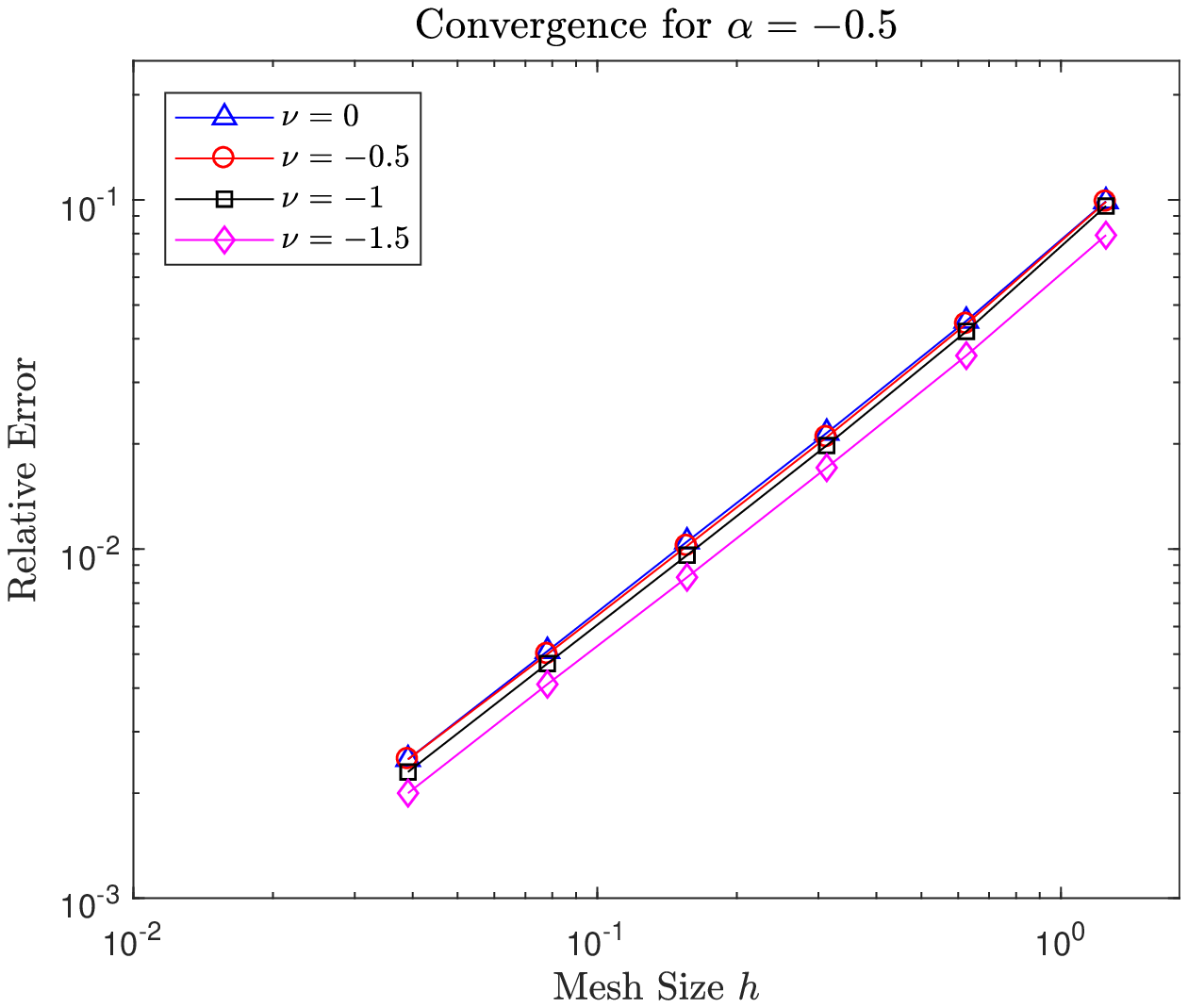}
\end{minipage}
\hspace{0.5cm}
\begin{minipage}[b]{0.45\linewidth}
\centering
\includegraphics[scale=0.5]{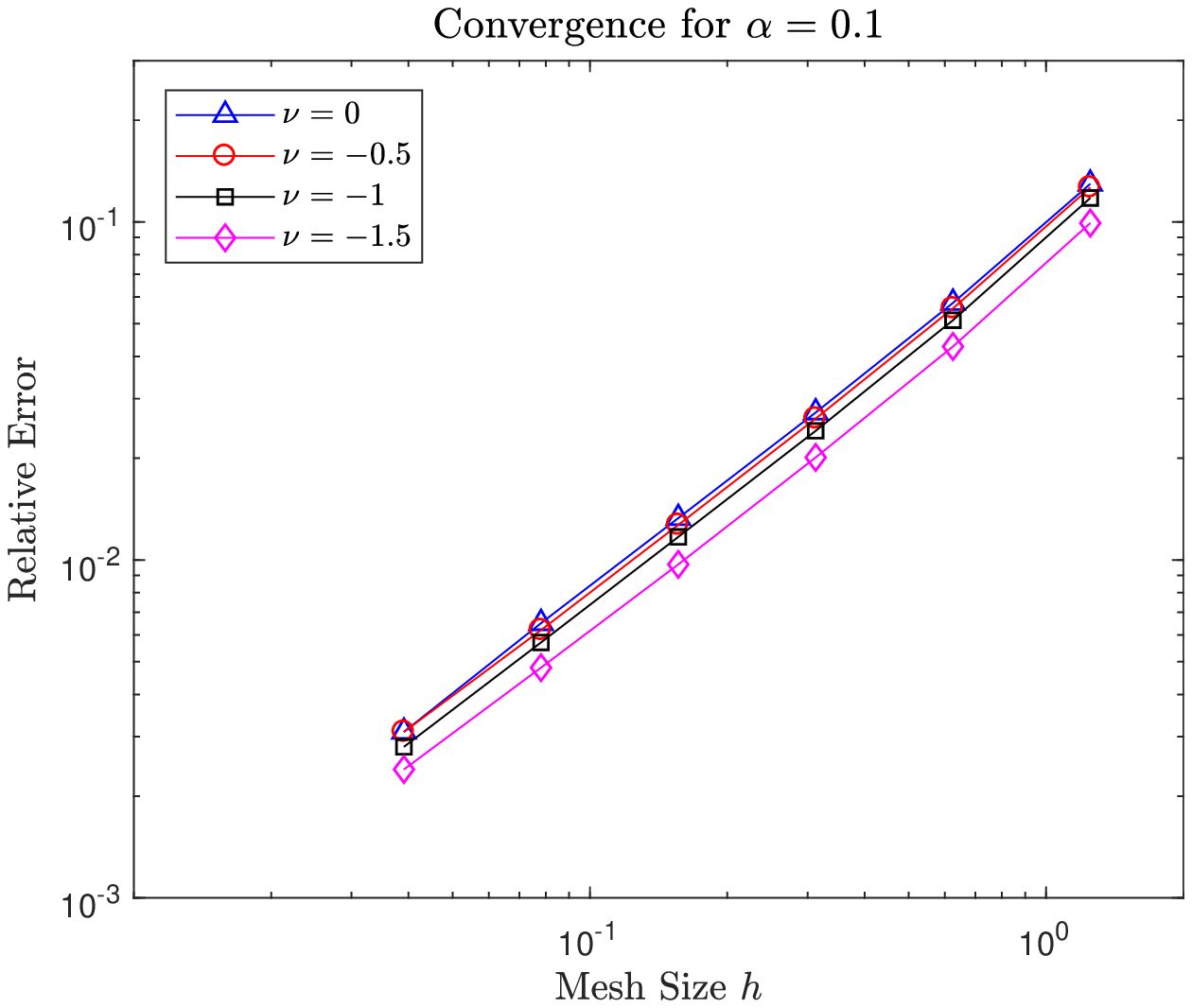}\\\hspace{0mm}\\
\includegraphics[scale=0.5]{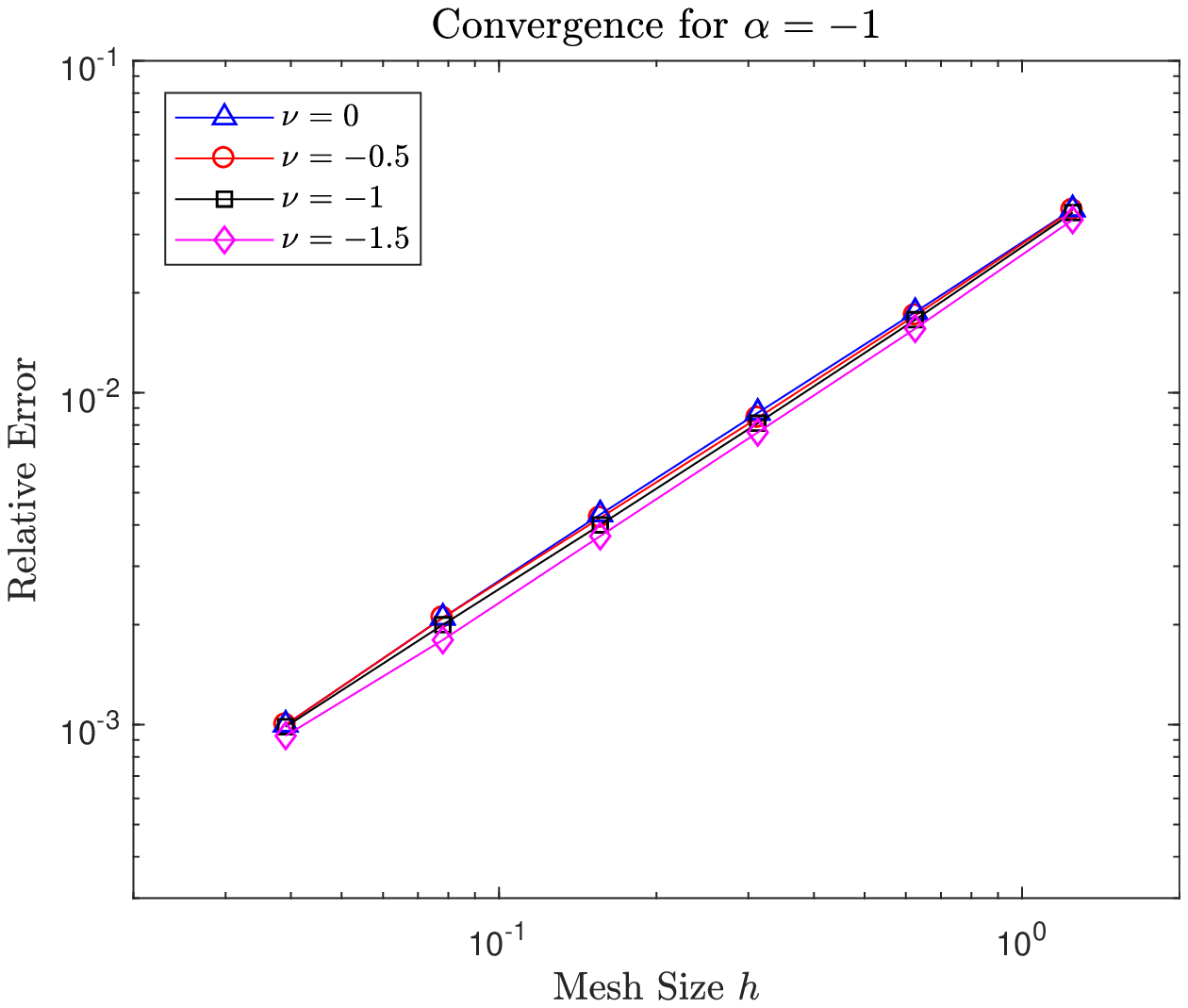}
\end{minipage}
\centering{
\includegraphics[scale=0.5]{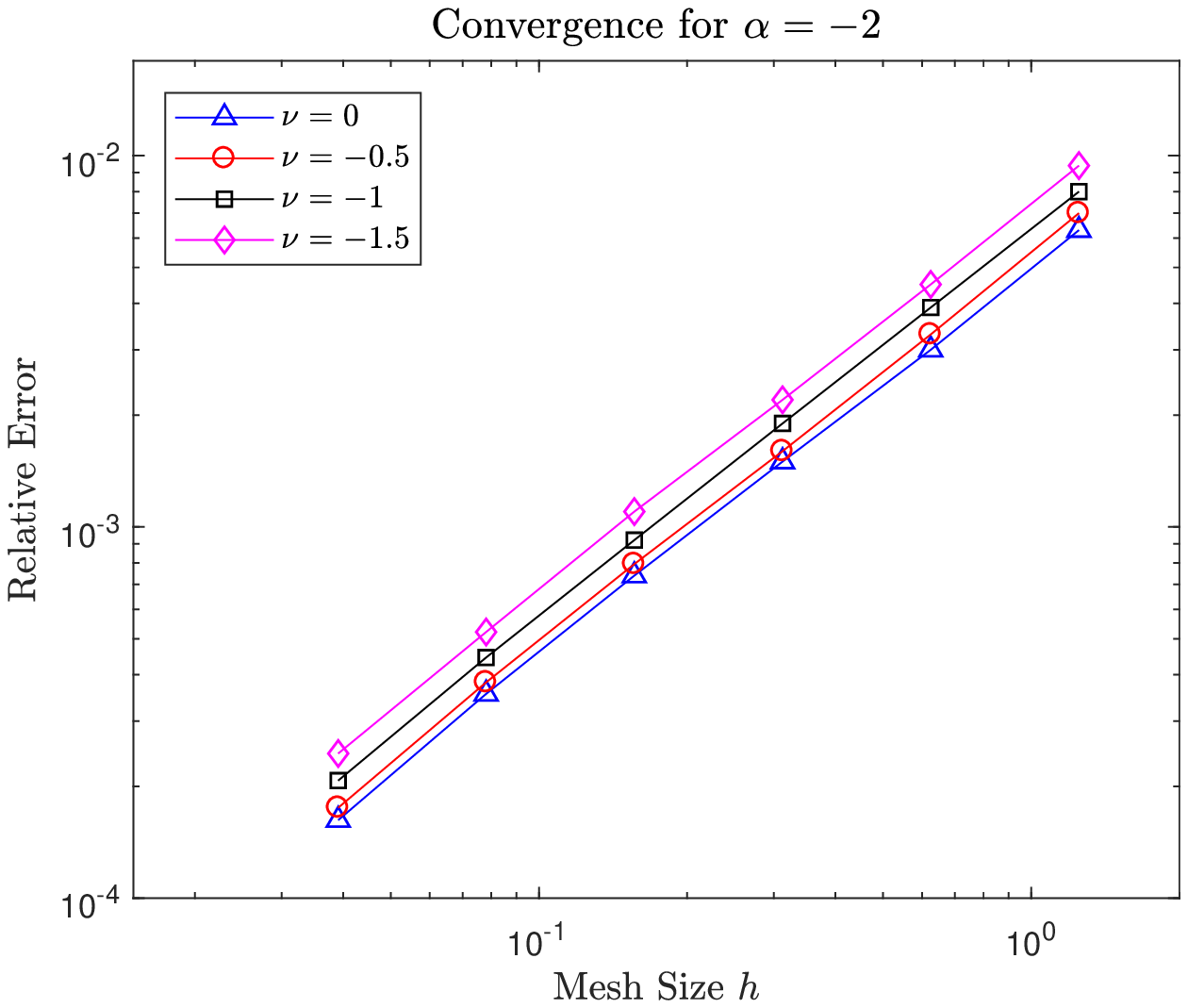}
\caption{Numerical convergence for $\alpha=0.5,0.1,-0.5,-1$ and $-2$.} \label{figure705}}
\end{figure}

It can be easily verified that conditions \eqref{equation303} and \eqref{equation304} are satisfied by these choices. The parameters $N$ and $R$ were set at $5$ and $15$ respectively, and an initial state assumed, with
\[
c_0(x)=\left\{ \begin{aligned}
\hspace{2mm} 1\hspace{4mm} &\textup{for} \hspace{2mm} 5<x<15, \\
\hspace{2mm} 0\hspace{4mm} &\textup{for} \hspace{2mm} x\geq 15,
\end{aligned} \right.
\]
and $d_0$ being the $N-$vector consisting entirely of $1'$s. The parameters $\alpha$ and $\nu$ were varied, taking all possible combinations of $\alpha\in\left\lbrace0.5,0.1,-0.5,-1,-2\right\rbrace$ and $\nu\in\left\lbrace0,-0.5,-1,-1.5\right\rbrace$, with the final time $T$ selected in each case to allow the system to reach a near equilibrium state.

The approximate solutions generated by the numerical scheme were compared to the exact solutions derived in \cite[Section 7.1]{baird18}, with the discrepancy being measured by taking the relative error with respect to the norm on $L_1\left([0,T),X_D\right)\times L_1\left([0,T),X_C^R\right)$. That is, supposing $u^h=(\underline{u}_D^h(t),u_C^h(x,t))$ is our approximation of an exact solution $u=(\underline{u}_D(t),u_C^R(x,t))$, then we measure the error via
\begin{equation*}
	\textrm{Error}(u^h|u)=\frac{\|u^h-u\|}{\|u\|},
\end{equation*}
where the norm $\|\cdot\|$, is given by
\begin{align*}
\|u\|&=\int_0^T\sum_{i=1}^N i \left|u_{Di}(t)\right|\dd t+\int_0^T\int_N^R\left|u_C^R(x,t)\right|\,x\dd x\dd t.
\end{align*}
For each model configuration, we computed approximate solutions over a sequence of uniform meshes, refining at each step by halving the mesh parameter $h$. The charts in Figure \ref{figure705} plot the observed relative error against the mesh parameter $h$, for all possible parameter configurations. From even the briefest examination of the charts it is clear that as the mesh is refined, the relative error of the approximations is reduced. Whilst if we were to examine the gradients of the lines appearing in Figure \ref{figure705}, then they appear generally to be getting closer to 1, as the mesh is refined. With the gradients between the most refined mesh pairings having a mean value of 1.0301, across all configurations. This would suggest that our numerical scheme has order $\gamma\approx1$, with the error in the approximations being $\mathcal{O}(h)$. The full numerical details of the errors and the associated convergence rates underlying Figure \ref{figure705} may be found in  \cite[Appendix A]{bairdphd}.

\section{Conclusions}

In this article we introduced a numerical scheme for the approximate solution of
a truncated version of a mixed discrete-continuous fragmentation model. The scheme was
based upon a finite volume discretisation of the modelling equation for the continuous component.

The resulting numerical approximations were first shown to be nonnegative and
to conserve mass, provided the underlying mesh satisfied certain constraints.
Following which we established the weak convergence of a subsequence of our
approximations, as the mesh size parameter $h$ was decreased to zero. The resulting limits
were then shown to provide a weak solution to the truncated model.

By relating the scalar-valued weak formulation of our truncated model to an
equivalent weak formulation within a Banach space setting, we were able to establish
a one-to-one relationship between any scalar and Banach-space-valued
weak solutions. Under suitable constraints, these Banach-space-valued weak solutions
were shown to provide the unique strong solution to the associated abstract Cauchy
problem, in the process establishing the uniqueness of the original scalar-valued
weak solutions. Additionally, as a further consequence of this linkage, the scalar
weak solutions were shown to be differentiable classical solutions.

Finally, by conducting a range of experiments with a test model, under varying model
parameter choices and mesh refinements, we experimentally established that the
error in our numerical solutions was $\mathcal{O}(h)$.

\section*{Acknowledgments}
\noindent This work was supported by the UK Engineering and Physical Sciences Research Council [EP/J500495/1  03].

\bibliographystyle{elsarticle-num}
\bibliography{references2}

\end{document}